\documentclass{article}
\usepackage{amsthm, amsmath, amssymb, inputenc, authblk, comment, tikz, array, longtable, enumerate, comment, xcolor}

\newtheorem{theorem}{Theorem}[section]
\newtheorem{claim}{}[theorem]
\newtheorem{lemma}[theorem]{Lemma}
\theoremstyle{definition}
\newtheorem{definition}[theorem]{Definition}

\DeclareMathOperator{\ex}{ex}

\tikzset{every node/.style={inner sep=1.5pt,circle,draw,fill}}

\tikzstyle{every path}=[thick]

\newdimen\R
\newcommand{\cG}{\mathcal G}
\newcommand{\cH}{\mathcal H}
\newcommand{\cB}{\mathcal B}
\newcommand{\cP}{\mathcal P}
\newcommand{\cF}{\mathcal F}

\title{Planar Tur\'an number of the $7$-cycle}

\author[1]{Ruilin Shi}
\author[1]{Zach Walsh}
\author[1]{Xingxing Yu}

\affil[1]{School of Mathematics, Georgia Institute of Technology, USA}

\date{\today}

\begin{document}

\maketitle

\begin{abstract}
The \emph{planar Tur\'an number} $\ex_{\cP}(n,H)$ of a graph $H$ is the maximum number of edges in an $n$-vertex planar graph without $H$ as a subgraph. Let $C_{\ell}$ denote the cycle of length $\ell$. The planar Tur\'an number $\ex_{\cP}(n,C_{\ell})$ behaves differently for $\ell\le 10$ and for $\ell\ge 11$, and it is known when $\ell \in \{3,4,5,6\}$.
We prove that $\ex_{\cP}(n,C_7) \le \frac{18n}{7} - \frac{48}{7}$ for all $n > 38$, and show that equality holds for infinitely many integers $n$.
\end{abstract}

\section{Introduction} \label{sec:introduction}
The \emph{Tur\'an number} $\ex(n,H)$ of a graph $H$ is the maximum number of edges in an $n$-vertex graph without $H$ as a subgraph.
Tur\'an's theorem \cite{Turan}, a cornerstone of extremal graph theory, states that $\ex(n, K_t) \le (1 - \frac{1}{t-1}){n\choose 2}$ for all integers $t$ and $n$ with $n,t \ge 3$, with equality for balanced complete $(t-1)$-partite graphs.
This has led to a huge amount of related work, including the Erd\H os-Stone theorem \cite{Erdos-Stone}, generalized Tur\'an numbers (see \cite{ALON2016146}), and Tur\'an problems for hypergraphs (see \cite{Keevash}).

Another well-studied variant of Tur\'an's theorem involves placing restrictions on the host graph; for example, forcing the host graph to be a hypercube (see \cite{hypercube}) or an Erd\H os-R\'enyi random graph (see \cite{random}).
In 2015, Dowden considered the variant in which the host graph is planar, and defined the \emph{planar Tur\'an number} $\ex_{\cP}(n,H)$ of a graph $H$ to be the maximum number of edges in an $n$-vertex planar graph without $H$ as a subgraph \cite{Dowden2015ExtremalCP}.
For example, it follows from Euler's formula that $\ex_{\cP}(n, C_3) = 2n - 4$ for all $n \ge 3$, where $C_{\ell}$ denotes the cycle of length $\ell$.
Since cycles are arguably the most natural choice for $H$ in the planar setting, Dowden proved that $\ex_{\cP}(n, C_4) \le \frac{15(n-2)}{7}$ for all $n \ge 4$ and $\ex_{\cP}(n, C_5) \le \frac{12n - 33}{5}$ for all $n \ge 11$, and showed that in both cases equality holds infinitely often \cite{Dowden2015ExtremalCP}.
This initiated a flurry of research on planar Tur\'an numbers.
We direct the reader to the survey paper of Lan, Shi, and Song \cite{Survey} for details regarding recent related work, while we continue to focus on the planar Tur\'an number of $C_{\ell}$.

In \cite{6-cycle}, Ghosh, Gy\H ori, Martin, Paulos, and Xiao took the next step and proved that $\ex_{\cP}(n, C_6) \le \frac{5n}{2} - 7$ with equality holding infinitely often, improving upon a result of Yan, Shi, and Song \cite{Theta-free}.
They also conjectured that $\ex_{\cP}(n, C_{\ell}) \le \frac{3(\ell - 1)}{\ell}n - \frac{6(\ell + 1)}{\ell}$ for all $\ell \ge 7$ and all sufficiently large $n$, as this is the correct bound for $\ell \in \{5,6\}$.
If true, this bound would be tight infinitely often for every $\ell \ge 7$.
However, this conjecture was disproved for all $\ell \ge 11$ by Cranston, Lidick\'y, Liu, and Shantanam \cite{CLLScounterexample}, and later also by Lan and Song \cite{ImprovedLowerBound}, using the fact that planar triangulations with at least $11$ vertices need not be Hamiltonian.
While the correct value of $\ex_{\cP}(n, C_{\ell})$ is rather mysterious for $\ell \ge 11$ (see \cite{CLLScounterexample} and \cite{ImprovedLowerBound}), we prove the conjecture of Ghosh et al. in the case that $\ell = 7$ (writing $e(G)$ for the number of edges of a graph $G$).

\begin{theorem} \label{main}
Let $n$ be an integer with $n > 38$ and let $G$ be an $n$-vertex $C_7$-free planar graph.
Then $e(G) \le \frac{18}{7}n - \frac{48}{7}$, and equality holds infinitely often.
\end{theorem}

The condition that $n > 38$ is necessary, due to the $38$-vertex graph obtained by gluing $18$ copies of $K_4$ together on a common edge.

We next describe, for each $\ell \ge 5$ and for infinitely many $n$, a $C_{\ell}$-free $n$-vertex planar graph with $\frac{3(\ell - 1)}{\ell}n - \frac{6(\ell + 1)}{\ell}$ edges, following Cranston et al. \cite{CLLScounterexample}.
Let  $G$ be an $n$-vertex planar graph with girth $\ell + 1$,  each vertex having degree 2 or $3$, and $\frac{\ell+1}{\ell - 1}(n - 2)$ edges; such a graph exists for infinitely many integers $n$ \cite[Lemma 2]{CLLScounterexample}.
Let $G'$ be obtained from $G$ by \emph{substituting} an $(\ell - 1)$-vertex planar triangulation $B$ for each vertex of $G$, which means that each vertex $v$ of $G$ is replaced by a copy of $B$ and $\deg_G(v)$ vertices of $B$ on a facial triangle are identified with the neighbors of $v$ in $G$.
Then $G'$ is a $C_{\ell}$-free $n'$-vertex graph with $\frac{3(\ell - 1)}{\ell}n' - \frac{6(\ell + 1)}{\ell}$ edges \cite[Corollary 5]{CLLScounterexample}.
In particular, if $\ell = 7$ then $G'$ is a $C_7$-free $n'$-vertex graph with $ \frac{18}{7}n' - \frac{48}{7}$ edges, so the bound in Theorem \ref{main} holds infinitely often.

To prove Theorem \ref{main} it suffices to consider graphs without small separations, and without small sets of vertices with only a few incident edges.
For a graph $G$ and constant $\alpha>0$, a set $S\subseteq V(G)$ is {\it $\alpha$-sparse} if $G$ has at most $\alpha |S|$ edges with at least one incident vertex in $S$.
For each positive integer $n$, we write $\cP_n$ for the class of $n$-vertex, $2$-connected, $C_7$-free plane graphs with no (18/7)-sparse set of order at most 4.
We will obtain Theorem \ref{main} as a consequence of the following result, by generating all $C_7$-free planar graphs from $\cup_{i\ge 1}\cP_i$ through several basic operations.

\begin{theorem} \label{main good}
Let $G \in \cP_n$ with $n \ge 7$.
Then $e(G) \le \frac{18}{7}n - \frac{48}{7}$.
\end{theorem}

Our proof uses the same strategy employed by Ghosh et al. in \cite{6-cycle} to find $\ex_{\cP}(n, C_6)$, as we believe it is the natural approach.
We also use their terminology whenever possible, for consistency.
This terminology and the proof strategy are both described in Section \ref{sec:proof strategy}.
We also enhance their strategy with a result (Lemma \ref{lem:paths}) about long paths in near triangulations which we believe will be useful for other planar Tur\'an problems.

For a graph $G$, we use $V(G)$ and $E(G)$ to denote its vertex set and edge set, respectively, and often write $|G|$ for $|V(G)|$ and $e(G)$ for $|E(G)|$. For any $S\subseteq V(G)$, we use $G[S]$ to denote the subgraph of $G$ induced by $S$ and write $G-S$ for $G[V(G)\setminus S]$ (and write $G-s$ instead of $S=\{s\}$). For any subgraph $H$ of $G$, we write $G-H$ for $G-V(H)$. 
For a subgraph $B$ of a graph of $G$, 
a \emph{$B$-path} is a path in $G$ that has both ends in $B$ and is internally disjoint from $B$. A $B$-path of length 1 is also known as a {\it chord} of $B$. 
For any positive integer $k$, we let $[k]=\{1, \ldots, k\}$.   

If $G$ is a plane graph, then $F(G)$ denotes its set of faces. It is well known that if $G$ is a 2-connected plane graph then each member of $F(G)$ is bounded by a cycle. The {\it interior} of a cycle $C$ in a plane graph is defined to be the subgraph of $G$ consisting of all edges and vertices of $G$ contained in the closed disc of the plane bounded by $C$.  
Paths and cycles will be represented as a sequence of vertices such that consecutive vertices in the sequence are adjacent. 
For instance, $x_1x_2x_3\ldots x_k$ is a path of length $k-1$, and $x_1x_2x_3\ldots x_kx_1$ is a cycle of length $k$.
 For any distinct vertices $x,y$ in a graph $G$, we use $d(x,y)$ to denote the distance between $x$ and $y$ in $G$, and if $x,y$ are on a path $P$ then $xPy$ denotes the subpath of $P$ between $x$ and $y$. For any cycle $C$ in a plane graph and any two distinct vertices $x,y$ on $C$, we use $xCy$ to denote the subpath of $C$ from $x$ to $y$ in clockwise order.

\section{Proof strategy and triangular-blocks} \label{sec:proof strategy}
We now describe the proof strategy for Theorem \ref{main}, which is based on that of Ghosh et al \cite{6-cycle}.
If $G$ is a connected $n$-vertex plane graph with $e$ edges and $f$ faces, then $e \le \frac{18n}{7} - \frac{48}{7}$ is equivalent via Euler's formula to 
$$24f - 17e + 6n \le 0,$$
by replacing the 48 with $24(n-e+f)$.
We will decompose the graph $G$ into edge-disjoint subgraphs, and show that no subgraph contributes too much towards the left-hand side of the inequality.
Due to the extremal construction, it is natural to decompose $G$ into unions of facial triangles. For facial triangles $F$ and $F'$ of $G$ we say that $F \sim F'$ if there is a sequence $F = F_1, F_2, \dots, F_k = F'$ of facial triangles of $G$ so that for each $i \in [k-1]$, $F_i$ and $F_{i+1}$ share an edge. Clearly $\sim$ is an equivalence relation on the facial triangles of $G$.
This motivates the following key definition \cite{6-cycle}.

\begin{definition}
Let $G$ be a plane graph. 
For $e\in E(G)$, if $e$ is not contained in any facial triangle, then let $B(e)$ be the 2-vertex subgraph of $G$ induced by $e$; otherwise, let $B(e)$ denote the union of all facial triangles equivalent to some facial triangle containing $e$. 
We call $B(e)$ the \emph{triangular-block} containing $e$, and it is \emph{trivial} if $B(e)$ has just two vertices. We use $\cB(G)$ to denote the collection of all triangular-blocks of $G$.
\end{definition}

It is clear from the definition that if $B$ is a triangular-block of $G$ and $e_1$ and $e_2$ are distinct edges of $B$, then $B = B(e_1) = B(e_2)$.  
In particular, any two distinct triangular-blocks of $G$ are edge-disjoint.
Also, note that any redrawing of $G$ with the same set of faces has the same triangular-blocks.
However, note that a face of a triangular-block $B$ is not necessarily a face of $G$.
For example, the interior $4$-face of $B_{7,b}$ in Figure \ref{table:7-vertex blocks} need not be a face of $G$.

\begin{definition}
Let $B$ be a triangular-block of a plane graph $G$.
A \emph{hole} of $B$ is a face of $B$ that is not a face of $G$.
\end{definition}

Not every face of a triangular-block can be a hole.
We next make two useful observations about faces of triangular-blocks. 

\begin{lemma} \label{lem:interior path}
Let $G$ be a $2$-connected plane graph and let $B$ be a triangular-block of $G$ whose outer face is a hole.
Let $C$ be the outer cycle of $B$ and let $P$ be a $C$-path in $B$.
Then some edge of $P$ is in two facial triangles of $G$.
\end{lemma}
\begin{proof}
Let $P_1$ and $P_2$ be the two subpaths of $C$ for which $P_1 \cup P$ and $P_2 \cup P$ are cycles.
Since the outer face of $B$ is a hole, each facial triangle of $B$ that is a facial triangle of $G$ is in the interior of either $P_1 \cup P$ or $P_2 \cup P$.
If no edge of $P$ is in two facial triangles, then no facial triangle of $G$ in the interior of $P_1 \cup P$ is equivalent under $\sim$ to a facial triangle of $G$ in the interior of $P_2 \cup P$, a contradiction.
\end{proof}

The following is a useful consequence of Lemma \ref{lem:interior path}.

\begin{lemma} \label{lem:common vertex}
Let $B$ be a triangular-block of a $2$-connected plane graph $G$, and let $F_1$ and $F_2$ be facial cycles of $B$ that each have length at least four.
Then $F_1$ and $F_2$ share at most one vertex.
\end{lemma}
\begin{proof}
By redrawing $G$, we may assume that $F_1$ is the outer cycle of $G$.
If the statement is false, then there are vertices $a$ and $b$ in $V(F_1)\cap V(F_2)$ and an $a$-$b$ path $P$ in the interior of $F_1$ so that $E(P) \subseteq E(F_2)$.
But by Lemma \ref{lem:interior path} this is a contradiction.
\end{proof}

The proof strategy for Theorem \ref{main} is to show that no triangular-block of $G$ contributes too much towards the sum $24f - 17e + 6n$.
We next make several definitions, following Ghosh et al. \cite{6-cycle}, to describe how a triangular-block contributes to $24f - 17e + 6n$.
This is easy for edges: each triangular-block $B$ contributes $e(B)$ towards $e$.
Then $$\sum_{B \in \cB(G)}e(B) = e,$$ since triangular-blocks are pairwise edge-disjoint.
We next describe the contribution of a triangular-block $B$ to $n$.

\begin{definition}
Let $G$ be a plane graph and $v\in V(G)$. 
We write $\cB(G)_v=\{B\in \cB(G): v\in V(B)\}$ and say that $v$ is a \emph{junction vertex} if $|\cB(G)_v| > 1$.
We define $n(v)=1/|\cB(G)_v|$, and $n(B)=\sum_{v\in V(B)}n(v)$.  
\end{definition}

Thus, for an $n$-vertex plane graph $G$,  
$$\sum_{B \in \cB(G)}n(B) =\sum_{B\in \cB(G)}\sum_{v\in V(B)}n(v)=\sum_{v\in V(G)}\sum_{B\in \cB(G)_v}n(v)= n.$$

We next define the contribution of a triangular-block to the number of faces of a plane graph $G$.
To do this we must develop some notation concerning the faces incident with edges of a triangular-block $B$.

\begin{definition}
Let $B$ be a triangular-block of a $2$-connected plane graph $G$.  
A facial cycle $F$ of $G$ that is not a facial cycle of $B$ but shares edges with $B$ is a \emph{petal} of $B$. The petal $F$ is said to be \emph{leaky} if the graph $F \cap B$ is disconnected.
The subgraph of $G$ consisting of $B$ and all petals of $B$ is the \emph{flower} of $B$.
\end{definition}

Our next definition will help deal with a facial cycle that intersects a triangular-block in a path with a least two edges.

\begin{definition}
Let $G$ be a $2$-connected plane graph and let $F$ be a facial cycle of $G$.
A triple $x_1x_2x_3$ of consecutive vertices on $F$ is a \emph{bad cherry} if $x_1x_3 \in E(G)\setminus E(F)$ and $x_1$ and $x_3$ are junction vertices of a triangular-block of $G$ that also contains the triangle $x_1x_2x_3x_1$.
Let $F'$ be the cycle obtained from $F$ in the following manner: for each bad cherry $x_1x_2x_3$ of $F$, replace the path $x_1x_2x_3$ in $F$ with the path $x_1x_3$.
We say that $F'$ is the \emph{refinement} of $F$.
For convenience, let $F'=F$ if $F$ has no bad cherry.
\end{definition}

Note that the refinement $F'$ of a facial cycle $F$ in a plane graph $G$ is a cycle which has length at most $|F|-1$ except when $F'=F$ (i.e., $F$ has no bad cherry). 
Also, it is straightforward to check that if $G$ is $C_7$-free and $|F| \ge 8$, then $|F'| \ge 8$ as well; we freely use this fact.
We can now define the contribution of a triangular-block of a plane graph $G$ to the number of faces of $G$.

\begin{definition}
Let $G$ be a $2$-connected plane graph.
Let $B$ be a triangular-block of $G$ and let $\cP(B)$ be the set of petals of $B$.
Let $F\in \cP(B)$ 
and let $F'$ be the refinement of $F$. Define $f_F(B)=e(F\cap B)/e(F')$ 
if $F \cap B$ is not a bad cherry, and $f_F(B)=1/e(F')$
if $F \cap B$ is a bad cherry.
Finally, define $$f(B) = (\textrm{number of non-hole faces of $B$}) + \sum_{F \in \cP(B)}f_F(B).$$
\end{definition}

It is not hard to check that if $G$ is a $2$-connected plane graph 
then $|F(G)|=\sum_{B \in \cB(G)}f(B)$.

Now that we have defined the contribution of a triangular-block to edges, vertices, and faces, we can define the contribution of a triangular-block to $24f - 17e + 6n$.

\begin{definition}
Let $G$ be a $2$-connected plane graph, and let $B$ be a triangular-block of $G$.
We define 
$$g(B) = 24f(B) - 17e(B) + 6n(B).$$
\end{definition}

We comment that $g(B)$ will be independent of the planar drawing of $G$ we choose as all planar drawings of $G$ we use will have the same collection of faces and facial cycles. 
So when we compute an estimate for $g(B)$ we may work with a particular planar drawing of $G$ that is most convenient and therefore of $B$ as well.
Also note that if $G$ has $n$ vertices, $e$ edges, and $f$ faces, then $\sum_{B\in \cB(G)}g(B) = 24f - 17e + 6n$.
This leads to the following proof strategy for Theorem \ref{main}: show that each triangular-block $B$ of $G$ satisfies $g(B) \le 0$.
However, there are two exceptional cases, shown in Figure \ref{exceptions}, for which $g(B) > 0$.
To deal with these cases we will instead find a partition of $\cB(G)$ so that the sum of $g(B)$ for the triangular-blocks in each part of the partition is at most zero.
Then Theorem \ref{main} will follow from the following lemma, which was also proved in \cite{6-cycle}.

\begin{lemma} \label{lem:partition}
Let $G$ be a $2$-connected $n$-vertex plane graph.
If there is a partition $\cP = \{\cB_1, \cB_2, \dots, \cB_m\}$ of $\cB(G)$ so that $\sum_{B \in \cB_i} g(B) \le 0$ for all $i \in [m]$, then $e(G) \le \frac{18n}{7} - \frac{48}{7}$.
\end{lemma}
\begin{proof}
The statement follows from the facts that $\sum_{B \in \cB(G)} g(B) = 24f(G) - 17e(G) + 6n$, and $e(G) \le \frac{18n}{7} - \frac{48}{7}$ if and only if $24f(G) - 17e(G) + 6n \le 0$.
\end{proof}

In order to apply this lemma we must characterize the possible triangular-blocks of a $C_7$-free plane graph and then estimate $g(B)$ for each such triangular-block $B$.
In Section \ref{sec:near triangulations} we prove some properties about paths in triangular-blocks, and in Section \ref{sec:C_7-free blocks} we characterize the possible triangular-blocks of a $C_7$-free plane graph.
Then in Sections \ref{sec:large blocks}-\ref{sec:small blocks} we bound $g(B)$ for each possible $C_7$-free triangular-block $B$.
We prove Theorem \ref{main good} in Section \ref{sec:2-connected case} and then prove Theorem \ref{main} in Section \ref{sec:main proof}.
We finish by discussing some related open problems in Section 
\ref{sec:future work}.

\section{Near Triangulations} \label{sec:near triangulations}

A \emph{near triangulation} is a plane graph in which every face is bounded by a triangle except possibly the outer face. 
Note that any near-triangulation is an iterative union of facial triangles, and could therefore be a triangular-block of a plane graph.
For convenience, we allow any planar embedding of $K_2$ to be a near triangulation with itself as its outer cycle. 
In this section we prove two lemmas that will help us characterize all possible $C_7$-free triangular-blocks of a plane graph, and also help us bound $g(B)$ for a given triangular-block $B$.

 \begin{lemma} \label{lem:paths}
Let $G$ be a near triangulation with outer cycle $C$, and let $x,y\in V(C)$ be distinct. Suppose $G$ contains an $x$-$y$ Hamiltonian path. Then $G$ also contains an $x$-$y$ path of length $\ell$ for all $\ell$ with $d(x,y)\le \ell\le |V(G)|-1$. 
\end{lemma}

\begin{proof}
We apply induction on $n:=|V(G)|\ge 2$. The assertion is easily seen to be true for $n=2$ and $n=3$. So assume $n\ge 4$. 

\medskip

{\it Case} 1. $G-x$ is 2-connected. 

Then $G-x$ is also a near triangulation. Let $D$ denote the outer cycle of $G-x$.  Let $x_1, x_2, \ldots, x_k$ be the neighbors of $x$ and assume that  $x_1, x_2, \ldots, x_k, y$ occur on $D$ clockwise in the order listed.  Then $x_1, x_k\in V(C)$, $x_1Dx_k=x_1x_2\ldots x_k$, and $x_kCx_1=x_kDx_1$. 

Without loss of generality, we may assume that some $x$-$y$ Hamiltonian  path in $G$ uses the edge $xx_t$ for some $t\in [k]$. Then $G-x$ has a Hamiltonian $x_t$-$y$ path. Hence, by induction, $G-x$ contains an $x_t$-$y$ path of any length between $d_{G-x}(x_t,y)$ and $n-2$. So by adding the edge $xx_t$, we see that $G$ has an $x$-$y$ path of any length between $d_{G-x}(x_t,y)+1$ and $n-1$. We are done if $d_{G-x}(x_t,y)\le d_G(x,y)-1$. So assume $d_{G-x}(x_t,y)\ge  d_G(x,y)$.

Let $P$ be a shortest $x$-$y$ path in $G$, i.e., the length of $P$ is $d_G(x,y)$, and let $xx_s\in E(P)$, where $s\in [k]$. Then for all $i\in [k]\setminus \{s\}$, $x_i\notin V(P)$; otherwise, $xx_i\cup x_i P y$ would be an $x$-$y$ path in $G$ shorter than $P$. Let $Q$ denote the subpath of $x_1Dx_k$ between $x_s$ and $x_t$. 
Then $Q\cup P$
is an $x_t$-$y$ path in $G-x$  of length $d_G(x,y)-1+|s-t|$. 
Hence, $d_{G-x}(x_t,y) \le d_G(x,y)-1 + |s-t|$. 
Thus, $G$ has an $x$-$y$ path of any length between $d_{G}(x,y)+|s-t|$ and $n-1$. 

Therefore, it remains to show that $G$ has $x$-$y$ paths of any length between $d_G(x,y)$ and $d_G(x,y)-1+|s-t|$. 
For each $x_i\in V(Q)$ we see that $P_i:=xx_i\cup x_iQx_s\cup (P-x)$ is an $x$-$y$ path in $G$ of length $d_G(x,y)+|i-s|$. Thus, the paths $P_i$, for all $x_i\in V(Q)$, are $x$-$y$ paths of all possible lengths between $d_G(x,y)$ and $d_G(x,y)+|s-t|$.

\medskip
{\it Case} 2. $G-x$ is not 2-connected. 

Then, since $G$ is a near triangulation, $C$ has chords from $x$ to $C-x$. Since $G$ has a Hamiltonian $x$-$y$ path, the chords of $C$ from $x$ must all end on $xCy-\{x,y\}$ or all end on $yCx-\{x,y\}$. By symmetry, we may assume that all chords of $C$ from $x$ end on $xCy$, and let $xz$ be the chord of $C$ with $zCy$ minimal. 

Let $C_1:=xCz\cup zx$ and $C_2:=zCx\cup xz$, which are two cycles in $G$. For $i\in [2]$, let $G_i$ denote the interior of $C_i$. Then each $G_i$ is a near triangulation and $C_i$ is its outer cycle. Observe that any $x$-$y$ Hamiltonian path in $G$ gives rise to an $x$-$z$ Hamiltonian path in $G_1$, as well as a $z$-$y$ Hamiltonian path in $G_2-x$ (and hence a $x$-$y$ Hamiltonian  path in $G_2$ by adding the edge $xz$). Moreover, $d_{G_2}(x,y)=d_G(x,y)$.  Let $n_i:=|V(G_i)|$ for $i\in [2]$.

By induction, $G_1$ has an $x$-$z$ path of length $\ell_1$ for every $\ell_1$ between $d_{G_1}(x,z)=1$ and $n_1-1$, and $G_2$ has an $x$-$y$ path of length $\ell_2$ for every $\ell_2$ between $d_{G_2}(x,y)=d_G(x,y)$ and $n_2-1$. 
It remains to show that $G$ has $x$-$y$ paths of any length between $n_2$ and $n-1$. 

Let $P$ be a Hamiltonian $z$-$y$ path in $G_2-x$. Thus, the length of $P$ is $n_2-2$. Now take an $x$-$z$ path $P_m$ in $G_1$ of length $m$ for each $m\in [n_1-1]$. Then  for $m\in [n_1-1]$, $P\cup P_m$ is an $x$-$y$ path of length $m+n_2-1$. Noting that $n=n_1+n_2-1$, we see that $G$ has an $x$-$y$ path of any length between $n_2$ and $n-1$.
\end{proof}

\begin{figure} 
\begin{center}
\begin{tabular}{ccc}
\begin{tikzpicture}[ultra thick]
     \node[label = above: {$y$}] (1) at (0:\R) {};
     \node (2) at (60:\R) {};
     \node (3) at (120:\R) {};
     \node (4) at (180:\R) {};
     \node[label = left: {$x$}] (5) at (240:\R) {};
      \node (6) at (300:\R) {};
     \draw (1) -- (2);
\draw (2) -- (3);
\draw (3) -- (4);
\draw (4) -- (5);
\draw (5) -- (6);
\draw (6) -- (1);
\draw (3) -- (5);
\draw (3) -- (6);
\draw (3) -- (1);
\end{tikzpicture} 
&
\begin{tikzpicture}[ultra thick]
     \node[label = {}] (1) at (0:\R) {};
     \node[label = right: {$y$}] (2) at (60:\R) {};
     \node (3) at (120:\R) {};
     \node (4) at (180:\R) {};
     \node[label = left: {$x$}] (5) at (240:\R) {};
      \node (6) at (300:\R) {};
     \draw (1) -- (2);
\draw (2) -- (3);
\draw (3) -- (4);
\draw (4) -- (5);
\draw (5) -- (6);
\draw (6) -- (1);
\draw (3) -- (5);
\draw (5) -- (1);
\draw (3) -- (1);
\end{tikzpicture}
&
\begin{tikzpicture}
     \node[label = {}] (1) at (18:\R) {};
     \node (2) at (90:\R) {};
     \node (3) at (162:\R) {};
     \node[label = left: {$x$}] (4) at (234:\R) {};
     \node[label = right: {$y$}] (5) at (306:\R) {};
     \node (6) at (0:0) {};
     \draw (1) -- (2);
\draw (2) -- (3);
\draw (3) -- (4);
\draw (4) -- (5);
\draw (5) -- (1);
\draw (2) -- (4);
\draw (2) -- (6);
\draw (2) -- (5);
\draw (4) -- (6);
\draw (5) -- (6);
\end{tikzpicture}

\\

$B_{6,a}$ & $B_{6,c}$ & $B_{6,d}$
\end{tabular}
\end{center}
\caption{The three exceptional cases for Lemma \ref{lem:Hpath}.}
\label{Hpath exceptions}
\end{figure}
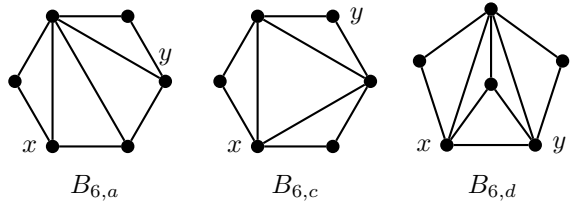

Next, we consider near triangulations with at most six vertices. 
Such a graph is certainly $C_7$-free, and is therefore relevant for the proof of Theorem \ref{main}.

\begin{lemma}\label{lem:Hpath}
Let $G$ be a near triangulation with outer cycle $C$, and let $x,y\in
V(C)$ be distinct. Suppose $|V(G)|\le 6$. Then $G$ contains an $x$-$y$ Hamiltonian
path or one of the following holds:
\begin{itemize}
\item [$(i)$] $xy$ is a chord of $C$.
\item [$(ii)$] $G$ has a path $xzy$ such that both $xz$ and $zy$ are
  chords of $C$, and $G-\{x,y,z\}$ consists of 3 isolated vertices (so
  $G\cong B_{6,a}$ or $G\cong B_{6,d}$ with $x,y$ as in Figure 1). 
\item [$(iii)$] $G\cong B_{6,c}$ with $x$, $y$ as in Figure 1.  
\end{itemize}
\end{lemma}

\begin{proof} Suppose for a contradiction that the assertion is false,
  and let $G$ be a counterexample with $|V(G)|$ minimum. Then $|V(G)|\ge
 4$, as otherwise $G$ contains an $x$-$y$ Hamiltonian path.
Moreover,  $xy$ is not a chord of $C$; as otherwise we
have $(i)$.

Observe that if $G$ contains a chord from $x$ to $xCy-x$ and a 
chord from $x$ to $yCx-x$, then since $|V(G)|\le 6$ (and because of
planarity), $G$ has no chord from $y$ to $xCy-y$ or $yCx-y$. Hence, we
may assume by symmetry (between $x$ and $y$ and between $xCy$ and
$yCx$) that $G$ has no chord from $x$ to $yCx-x$.

Since $xy$ is not a chord of $C$ and $G$ is a near triangulation, $y$
is not a cut vertex of the connected graph $G-x$. So let $Y$ denote the unique block of $G-x$ containing
$y$. Let $x_1,x_2\in V(Y)\cap V(C-x)$ such that $x,x_1,y,x_2$ occur in
$C$ in clockwise order, and $x_1Cx_2=Y\cap C$.  Since $G$ has no chord from $x$ to $yCx-x$, we see that
$x_2x\in E(C)$. Let $X:=G-(Y-x_1)$. Note that if $|V(X)|\ge 3$ then
$X$ is a near triangulation with outer cycle $C_X:=xCx_1\cup x_1x$, and
that if $|V(Y)|\ge 3$ then $Y$ is a near triangulation and its outer
cycle, denoted $C_Y$,  satisfies $x_1Cx_2=x_1C_Yx_2$.

We may assume $y\ne x_1$. For otherwise, $V(X)=\{x,y\}$ as $xy$
is not a chord of $C$. But then $x_1$ and $x_2$ are symmetric, and we
can relabel $x_1,x_2$ and flip the embedding of $G$. 

Now suppose $|V(G)|\le 5$.  Then $|V(X)|\le 4$ and
$|V(Y)|\le 4$. Hence, $X$ has an $x$-$x_1$ Hamiltonian path (since
$xx_1$ is not a chord of $C_X$) and, thus,
$Y$ has no $x_1$-$y$ Hamiltonian path. 
Hence, $Y$ consists of a
4-cycle $x_1aybx_1$ and chord $x_1y$, and we have $V(X) = \{x,x_1\}$. Choose the labeling so that
$x_1ay\subseteq x_1Cx_2$. Since $x_2x\in E(C)$ and $G$ is a near
triangulation, $xb\in E(G)$. But then $xbx_1ay$ is a $x$-$y$ Hamiltonian
path in $G$, a contradiction.

Thus, $|V(G)|=6$. Then $|V(X)|\le 5$. Hence, by the minimality of $G$,
$X$ has an $x$-$x_1$ Hamiltonian path (since
$xx_1$ is not a chord of $C_X$). Therefore, $Y$ has no $x_1$-$y$
Hamiltonian path. 
Thus $|V(X)|=3$ and $|V(Y)|=4$, or $|V(X)|=2$ and
$|V(Y)|=5$. Moreover, by the minimality of $G$,   $x_1y$ is a
chord of $C_Y$. If $|V(X)|=3$ and $|V(Y)|=4$ then $xx_1,x_1y$ are both chords of $C$, and $(ii)$
holds.

So $|V(Y)|=5$ and $|V(X)|=2$.  Note that, for any $w\in V(yC_Yx_1)\setminus
\{y,x_1\}$ with $w\in N(x)$, $Y$ has no $w$-$y$ Hamiltonian path, as
such a path and $xw$ would give an $x$-$y$ Hamiltonian path in $G$.
Thus, since $x_1y$ is a chord of $C_Y$, we have  $|V(C_Y)|=5$ and
$yC_Yx_1=yvx_2x_1$ is a path of length 3. 
Now $vx_1\notin E(G)$, as otherwise $x x_2 v x_1 \cup x_1 C_Y y$ is an $x$-$y$ Hamiltonian path in $G$, a contradiction.
Hence, $x_2y\in E(G)$, which shows that $(iii)$ holds. 
\end{proof}


\section{$C_7$-Free Triangular-Blocks} \label{sec:C_7-free blocks}

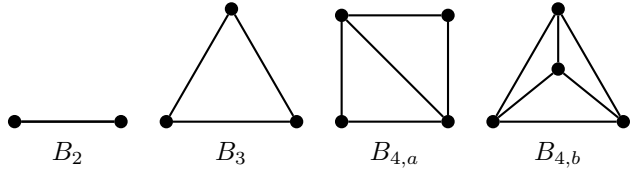
\begin{figure} 
\centering
\begin{tabular}{c c c c}
\begin{tikzpicture}
     \node[label = {}] (1) at (45:\R) {};
     \node (2) at (135:\R) {};
     \draw (1) -- (2);
\draw (2) -- (1);
\end{tikzpicture} 
&
\begin{tikzpicture}
     \node[label = {}] (1) at (90:\R) {};
     \node (2) at (210:\R) {};
     \node (3) at (330:\R) {};
     \draw (1) -- (2);
\draw (2) -- (3);
\draw (3) -- (1);
\end{tikzpicture}  

&

\begin{tikzpicture}
     \node[label = {}] (1) at (45:\R) {};
     \node (2) at (135:\R) {};
     \node (3) at (225:\R) {};
     \node (4) at (315:\R) {};
     \draw (1) -- (2);
\draw (2) -- (3);
\draw (3) -- (4);
\draw (4) -- (1);
\draw (2) -- (4);
\end{tikzpicture} 

&

\begin{tikzpicture}
     \node[label = {}] (1) at (90:\R) {};
     \node (2) at (210:\R) {};
     \node (3) at (330:\R) {};
     \node (4) at (90:\R/5) {};
     \draw (1) -- (2);
\draw (2) -- (3);
\draw (3) -- (1);
\draw (1) -- (4);
\draw (2) -- (4);
\draw (3) -- (4);
\end{tikzpicture}  \\

$B_2$ & $B_3$ & $B_{4,a}$ & $B_{4,b}$

\end{tabular}
\caption{Possible triangular-blocks on at most four vertices (up to redrawing)}
\label{table:small blocks}
\end{figure}

\begin{figure}
\centering
\begin{tabular}{c c c c}
\begin{tikzpicture}
     \node[label = {}] (1) at (18:\R) {};
     \node (2) at (90:\R) {};
     \node (3) at (162:\R) {};
     \node (4) at (234:\R) {};
     \node (5) at (306:\R) {};
     \draw (1) -- (2);
\draw (2) -- (3);
\draw (3) -- (4);
\draw (4) -- (5);
\draw (5) -- (1);
\draw (2) -- (4);
\draw (2) -- (5);
\end{tikzpicture}

&

\begin{tikzpicture}
     \node[label = {}] (1) at (45:\R) {};
     \node (2) at (135:\R) {};
     \node (3) at (225:\R) {};
     \node (4) at (315:\R) {};
     \node (5) at (0:0) {};
     \draw (1) -- (2);
\draw (2) -- (3);
\draw (3) -- (4);
\draw (4) -- (1);
\draw (2) -- (4);
\draw (1) -- (3);
\end{tikzpicture} 

&

\begin{tikzpicture}
     \node[label = {}] (1) at (45:\R) {};
     \node (2) at (135:\R) {};
     \node (3) at (225:\R) {};
     \node (4) at (315:\R) {};
      \node(5) at (45:\R/2) {};
     \draw (1) -- (2);
\draw (2) -- (3);
\draw (3) -- (4);
\draw (4) -- (1);
\draw (2) -- (4);
\draw (1) -- (5);
\draw (2) -- (5);
\draw (4) -- (5);
\end{tikzpicture} 

&

 \begin{tikzpicture}
     \node[label = {}] (1) at (90:\R) {};
     \node (2) at (210:\R) {};
     \node (3) at (330:\R) {};
     \node (4) at (90:0cm) {};
     \node (5) at (90:\R/2) {};
     \draw (1) -- (2);
\draw (2) -- (3);
\draw (3) -- (1);
\draw (1) -- (4);
\draw (2) -- (4);
\draw (3) -- (4);
\draw (1) -- (5);
\draw (2) -- (5);
\draw (3) -- (5);
\end{tikzpicture} \\

$B_{5,a}$ & $B_{5,b}$ & $B_{5,c}$ & $B_{5,d}$ 

\end{tabular}
\caption{Possible $5$-vertex triangular-blocks (up to redrawing)}
 \label{table:5-vertex blocks}
\end{figure}
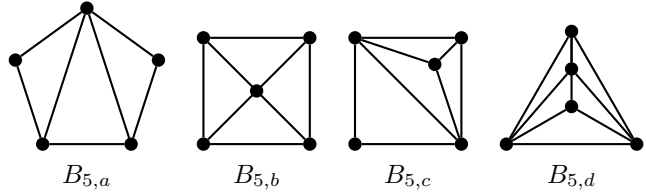

\begin{figure} 
\centering
\begin{tabular}{c c c}
\begin{tikzpicture}[ultra thick]
     \node[label = {}] (1) at (0:\R) {};
     \node (2) at (60:\R) {};
     \node (3) at (120:\R) {};
     \node (4) at (180:\R) {};
     \node (5) at (240:\R) {};
      \node (6) at (300:\R) {};
     \draw (1) -- (2);
\draw (2) -- (3);
\draw (3) -- (4);
\draw (4) -- (5);
\draw (5) -- (6);
\draw (6) -- (1);
\draw (3) -- (5);
\draw (3) -- (6);
\draw (3) -- (1);
\end{tikzpicture}

& \begin{tikzpicture}[ultra thick]
     \node[label = {}] (1) at (0:\R) {};
     \node (2) at (60:\R) {};
     \node (3) at (120:\R) {};
     \node (4) at (180:\R) {};
     \node (5) at (240:\R) {};
      \node (6) at (300:\R) {};
     \draw (1) -- (2);
\draw (2) -- (3);
\draw (3) -- (4);
\draw (4) -- (5);
\draw (5) -- (6);
\draw (6) -- (1);
\draw (3) -- (5);
\draw (3) -- (6);
\draw (2) -- (6);
\end{tikzpicture}

& \begin{tikzpicture}[ultra thick]
     \node[label = {}] (1) at (0:\R) {};
     \node (2) at (60:\R) {};
     \node (3) at (120:\R) {};
     \node (4) at (180:\R) {};
     \node (5) at (240:\R) {};
      \node (6) at (300:\R) {};
     \draw (1) -- (2);
\draw (2) -- (3);
\draw (3) -- (4);
\draw (4) -- (5);
\draw (5) -- (6);
\draw (6) -- (1);
\draw (3) -- (5);
\draw (5) -- (1);
\draw (3) -- (1);
\end{tikzpicture} \\

$B_{6,a}$ & $B_{6,b}$ & $B_{6,c}$ \\

\hline \\
\begin{tikzpicture}
     \node[label = {}] (1) at (18:\R) {};
     \node (2) at (90:\R) {};
     \node (3) at (162:\R) {};
     \node (4) at (234:\R) {};
     \node (5) at (306:\R) {};
     \node (6) at (0:0) {};
     \draw (1) -- (2);
\draw (2) -- (3);
\draw (3) -- (4);
\draw (4) -- (5);
\draw (5) -- (1);
\draw (2) -- (4);
\draw (2) -- (6);
\draw (2) -- (5);
\draw (4) -- (6);
\draw (5) -- (6);
\end{tikzpicture}

& \begin{tikzpicture}
     \node[label = {}] (1) at (18:\R) {};
     \node (2) at (90:\R) {};
     \node (3) at (162:\R) {};
     \node (4) at (234:\R) {};
     \node (5) at (306:\R) {};
     \node (6) at (162:2\R/3) {};
     \draw (1) -- (2);
\draw (2) -- (3);
\draw (3) -- (4);
\draw (4) -- (5);
\draw (5) -- (1);
\draw (2) -- (4);
\draw (2) -- (6);
\draw (2) -- (5);
\draw (4) -- (6);
\draw (3) -- (6);
\end{tikzpicture}

& \begin{tikzpicture}
     \node[label = {}] (1) at (18:\R) {};
     \node (2) at (90:\R) {};
     \node (3) at (162:\R) {};
     \node (4) at (234:\R) {};
     \node (5) at (306:\R) {};
     \node (6) at (270:\R/3) {};
     \draw (1) -- (2);
\draw (1) -- (3);
\draw (2) -- (3);
\draw (3) -- (4);
\draw (4) -- (5);
\draw (5) -- (1);
\draw (1) -- (6);
\draw (3) -- (6);
\draw (4) -- (6);
\draw (5) -- (6);
\end{tikzpicture} \\

$B_{6,d}$ & $B_{6,e}$ & $B_{6,f}$ \\
\hline \\

\begin{tikzpicture}
     \node[label = {}] (1) at (45:\R) {};
     \node (2) at (135:\R) {};
     \node (3) at (225:\R) {};
     \node (4) at (315:\R) {};
     \node (5) at (45:\R/3) {};
     \node (6) at (45:2\R/3) {};
     \draw (1) -- (2);
\draw (2) -- (3);
\draw (3) -- (4);
\draw (4) -- (1);
\draw (2) -- (4);
\draw (1) -- (6);
\draw (2) -- (5);
\draw (2) -- (6);
\draw (4) -- (5);
\draw (4) -- (6);
\draw (5) -- (6);
\end{tikzpicture} 

& \begin{tikzpicture}
     \node[label = {}] (1) at (45:\R) {};
     \node (2) at (135:\R) {};
     \node (3) at (225:\R) {};
     \node (4) at (315:\R) {};
     \node (5) at (20:\R/2) {};
     \node (6) at (70:\R/2) {};
     \draw (1) -- (2);
\draw (2) -- (3);
\draw (3) -- (4);
\draw (4) -- (1);
\draw (2) -- (4);
\draw (1) -- (5);
\draw (1) -- (6);
\draw (2) -- (6);
\draw (4) -- (5);
\draw (4) -- (6);
\draw (5) -- (6);
\end{tikzpicture} 

& \begin{tikzpicture}
     \node[label = {}] (1) at (45:\R) {};
     \node (2) at (135:\R) {};
     \node (3) at (225:\R) {};
     \node (4) at (315:\R) {};
     \node (5) at (45:\R/2) {};
     \node (6) at (225:\R/2) {};
     \draw (1) -- (2);
\draw (2) -- (3);
\draw (3) -- (4);
\draw (4) -- (1);
\draw (2) -- (4);
\draw (1) -- (5);
\draw (2) -- (5);
\draw (4) -- (5);
\draw (3) -- (6);
\draw (2) -- (6);
\draw (4) -- (6);
\end{tikzpicture}  \\

$B_{6,g}$ & $B_{6,h}$ & $B_{6,i}$

\end{tabular}
\caption{Possible $6$-vertex triangular-blocks with a chord (up to redrawing)}
\label{table:6-vertex blocks}
\end{figure}
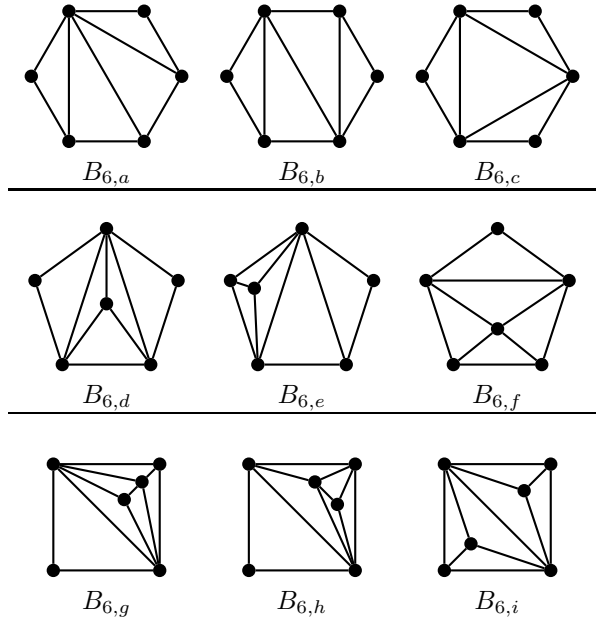

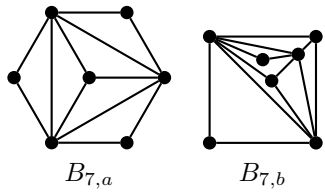
\begin{figure}
\centering
\begin{tabular}{c c}
\begin{tikzpicture}[ultra thick]
     \node[label = {}] (1) at (0:\R) {};
     \node (2) at (60:\R) {};
     \node (3) at (120:\R) {};
     \node (4) at (180:\R) {};
     \node (5) at (240:\R) {};
      \node (6) at (300:\R) {};
      \node (7) at (0:0) {};
     \draw (1) -- (2);
\draw (2) -- (3);
\draw (3) -- (4);
\draw (4) -- (5);
\draw (5) -- (6);
\draw (6) -- (1);
\draw (3) -- (5);
\draw (5) -- (1);
\draw (3) -- (1);
\draw (7) -- (1);
\draw (3) -- (7);
\draw (5) -- (7);
\end{tikzpicture}

& \begin{tikzpicture}
     \node[label = {}] (1) at (45:\R) {};
     \node (2) at (135:\R) {};
     \node (3) at (225:\R) {};
     \node (4) at (315:\R) {};
     \node (5) at (45:\R/6) {};
     \node (6) at (45:2\R/3) {};
     \node (7) at (90:2\R/5) {};
     \draw (1) -- (2);
\draw (2) -- (3);
\draw (3) -- (4);
\draw (4) -- (1);
\draw (2) -- (4);
\draw (1) -- (6);
\draw (2) -- (5);
\draw (2) -- (6);
\draw (4) -- (5);
\draw (4) -- (6);
\draw (5) -- (6);
\draw (2) -- (7);
\draw (6) -- (7);
\end{tikzpicture} \\

$B_{7,a}$ & $B_{7,b}$

\end{tabular}
\caption{Possible $7$-vertex $C_7$-free triangular-blocks (up to redrawing)}
 \label{table:7-vertex blocks}
\end{figure}

In this section we characterize all possible triangular-blocks in a $C_7$-free plane graph, up to redrawings that preserve all facial cycles. 
Certainly any near triangulation with at most six vertices is a candidate.
In fact, the following corollary of Lemma~\ref{lem:common vertex} shows that these are the only possible triangular-blocks with at most six vertices.

\begin{lemma} \label{lem:near-triangulation}
Let $G$ be a $2$-connected plane graph and let $B$ be a triangular-block of $G$ with at most six vertices.
Then $B$ is a near triangulation, up to redrawings that preserve all facial cycles.
\end{lemma}

Figures \ref{table:small blocks}, \ref{table:5-vertex blocks}, and \ref{table:6-vertex blocks} list all near triangulations with at most five vertices and all six-vertex near triangulations with a chord. 
We next prove that there are only two possible $C_7$-free triangular-blocks with more than six vertices; these are shown in Figure \ref{table:7-vertex blocks}.

\begin{lemma} \label{lem:all blocks}
Every $C_7$-free triangular-block of a $2$-connected plane graph $G$ has at most six vertices, with the exceptions of $B_{7,a}$ and $B_{7,b}$ as in Figure \ref{table:7-vertex blocks}.
\end{lemma}
\begin{proof} Let $B$ be a $C_7$-free triangular-block of $G$ with at least seven vertices.

Suppose $B$ contains a $7$-vertex triangular-block $B_{7,a}$ or $B_{7,b}$. 
Note that triangular faces of $B_{7,a}$ and $B_{7,b}$ are not holes, by Lemma \ref{lem:interior path}.
We may assume that $G$ is drawn so that the $B_{7,a}$ or $B_{7,b}$ in $B$ are drawn as in Figure \ref{table:7-vertex blocks}. If $B$ has more than seven vertices, then $B$ has a plane subgraph obtained by adding a vertex adjacent to both ends of an edge on a non-triangular facial cycle of $B_{7,a}$ or $B_{7,b}$.
It is straightforward to check that any graph of this form contains $C_7$, a contradiction.
Therefore $B \in \{B_{7,a} , B_{7,b}\}$.
Hence, we may assume that $B$ contains neither $B_{7,a}$ nor $B_{7,b}$. 

Note that $B$ has a subgraph $H$ that is obtained from a $6$-vertex triangular-block $B'$ by adding a vertex $z$ that is adjacent to both vertices of an edge $xy$ of some facial cycle of $B'$ bounding a hole of $B'$.

Since $B$ is $C_7$-free, $B'$ has no $x$-$y$ Hamiltonian path. 
By Lemma \ref{lem:near-triangulation}, $B'$ is a near triangulation, and Lemma \ref{lem:Hpath} implies that $B'$ is not a triangulation.
By redrawing $G$ we may assume that the outer cycle $F_1$ of $B'$ has length at least four.
Lemma \ref{lem:Hpath} implies that if $xy$ is an edge on the outer cycle of $B'$, then $B'\cong B_{6,d}$ and $H = B_{7,a}$, a contradiction.
So $xy$ is incident to an interior triangular hole of $B'$.
This implies that $H$ has another facial cycle, $F_2$, of length at least four. 
By Lemma \ref{lem:common vertex}, $|F_1|=|F_2|=4$ and $|V(F_1)\cap V(F_2)|=1$.
Let $F_1=abcda$ and $F_2=auvwa$.

Suppose that $F_1$ has no chord in its interior.
Note that each interior face of $H$ is triangular except for $F_2$.
If $ab$ and $bc$ are incident with the same interior triangular face of $H$, then $ac$ is a chord of $H$, a contradiction.
Since $ab$ and $bc$ are incident with different interior faces, there is an interior edge of $H$ with $b$ as an end and with the other end in the set $\{u,v,w\}$.
However, neither $bu$ nor  $bw$ is an edge of $H$; otherwise, $H$ (and hence $G$) would contain $C_7$.
So $bv$ is an edge of $H$.
By symmetry, $dv$ is an edge of $H$ and $du$ and $dw$ are not edges of $H$.
The $4$-cycles $abvua$ and $advwa$ each bound two triangular faces of $H$, 
which forces $av$ to be an edge in the interior of both $abvua$ and $advwa$. This is a contradiction.

Thus, we may assume that $F_1$ has a chord in its interior.
Figure \ref{table:6-vertex blocks} shows all possible $6$-vertex triangular-blocks whose outer cycle has length four and has a chord in its interior, and 
only $B_{6,g}$ and $B_{6,h}$ could have an interior hole.
This hole is bounded by a facial cycle $F$ of $B'$ containing both interior vertices of $B'$.
It is not hard to check that for any distinct $s,t\in V(F)$, $B'$ has an $s$-$t$ Hamiltonian path, with one exception  (when $B'\cong B_{6,g}$, $s \in V(F_1)$, and $t$ is of degree $4$ in $B'$) that results in $H = B_{7,b}$.
\end{proof}

We will also need the following characterization of bad cherries in facial cycles with length at most six.

\begin{lemma} \label{lem:refinements}
Let $G$ be a $2$-connected plane graph, let $F$ be a petal of some triangular-block $B$ of $G$ with $4 \le |F| \le 6$ and $|B| \ge 3$, and let $F'$ be the refinement of $F$.
Then $F = F'$, unless 
\begin{enumerate}[$(i)$]
    \item $B \in \{B_{5,d}, B_{6,g}, B_{6,h}, B_{7,b}\}$,
    \item $F = wx_1x_2x_3$ where $x_1x_2x_3$ is a bad cherry on the outer cycle of $B$, and
    \item $wx_1$ and $wx_3$ are trivial triangular-blocks incident with exactly one face of $G$ with length less than eight.
\end{enumerate}
\end{lemma}
\begin{proof}
Assume that $F \ne F'$.
Then $F$ has a bad cherry $x_1x_2x_3$.
By the definition of bad cherry, $x_1,x_2,x_3$ are in a common triangular-block $B$ of $G$. 
We may assume that $G$ is drawn so that $x_1x_2x_3$ is contained in the outer cycle $C$ of $B$ and that $F$ is not in the interior of $C$.
Let $B'$ be the interior of the triangle $x_1x_2x_3x_1$ in $B$, so $B'$ is a subgraph of $B$. 
Since $G$ has no (18/7)-sparse sets of order at most two, 
there are at least two vertices in $B'$, and therefore $B'$ has at least five vertices.
Thus, by Lemma \ref{lem:all blocks}, $B \in \{B_{5,d}, B_{6,g}, B_{6,h}, B_{7,b}\}$ and $(i)$ holds.
Note that $B' \in \{B_{5,d}, B_{6,g}, B_{6,i}\}$, up to redrawing.

We claim that $|F|=4$.
Since $B' \in \{B_{5,d}, B_{6,g}, B_{6,i}\}$, by checking cases, there is a length $i$ $x_1$-$x_3$ path $P_i$ in $B'$ for each $i \in [4]$.
Then $|F| \notin \{5,6\}$, or else $(F-x_2)\cup P_i$ with $i \in \{3,4\}$ is a $C_7$ in $G$. 
Therefore $|F| = 4$.

Let $w$ be the vertex in $V(F)\setminus \{x_1,x_2,x_3\}$.
We see that $(ii)$ holds.
If $wx_1$ or $wx_3$ is in a facial cycle of length less than eight other than $F$, then such a facial cycle contains an $x_1$-$x_3$ path of length $\ell \in \{3,4,5,6\}$ that is internally disjoint from $B'$.
The concatenation of this path with a length $(7 - \ell)$ $x_1$-$x_3$ path of $B'$ forms a $C_7$ in $G$, a contradiction.
Therefore $wx_1$ and $wx_3$ are trivial triangular-blocks of $G$ that are incident with exactly one face of $G$ with length less than eight and $(iii)$ holds.
\end{proof}


\section{Large Triangular-Blocks} \label{sec:large blocks}

Now that we have a complete characterization of possible triangular-blocks in a $C_7$-free plane graph, we proceed by estimating $$g(B)=24f(B)-17e(B)+6n(B)$$ for each possible triangular-block $B$.
In this section we analyze the possible triangular-blocks with at least six vertices for a graph $G \in \cP_n$.
We first consider chordless near triangulations, and then divide the remaining cases into three groups based on the length of the outer cycle.

\begin{lemma}
Let $G \in \cP_n$ with $n \ge 7$, and let $B$ be a triangular-block of $G$.
If $B$ is a $6$-vertex near triangulation whose outer cycle has no chord and whose outer face is its only hole, then $g(B) \le 0$.
\end{lemma}
\begin{proof}
Let $C$ be the outer cycle of $B$; this is the only possible non-triangle facial cycle of $B$. 
Since $C$ has a chord, $|C| \le 5$.
It follows from Lemmas \ref{lem:Hpath} and \ref{lem:paths} that for any distinct $x,y\in V(C)$, there is an $x$-$y$ path in $B$ of length $l$ for each $l$ between $d_B(x,y)$ (at most 2) and $|B|-2=4$. Thus, to avoid $C_7$, each petal $F$ of $B$ has length at least $8$ and, hence, the refinement of each petal has length at least $8$.

Note that at least two vertices on $C$ are junction vertices since $G$ is $2$-connected, and for each junction vertex $v$ we have $n(v)\le 1/2$. Moreover, from Euler's formula, we have $|E(B)|=15-|C|$ and $|F(B)|=11-|C|$.

If $|C| = 4$, then
$$g(B) \le 24(6 + 4/8) - 17\cdot 11 + 6(4 + 2/2) = -1;$$
and if $|C| = 5$, then 
$$g(B) \le 24(5 + 5/8) - 17\cdot 10 + 6(4 + 2/2) = -5.$$
Thus we may assume that $B$ is a triangulation.
If all vertices of $C$ are junction vertices of $B$ then
$$g(B) \le 24(7 + 3/8) - 17\cdot 12 + 6(3 + 3/2) = 0.$$
So assume that only two vertices of $C$ are junction vertices of $B$. Then $B$ has a bad cherry $x_1x_2x_3$ on a facial cycle $F$ of $B$ such that $x_1x_2,x_2x_3\in E(F)$. Thus, 
$$g(B) \le 24(7 + 1/8 + 1/8) - 17\cdot 12 + 6(4 + 2/2) = 0,$$
as desired.
\end{proof}

We next consider $6$-vertex and 7-vertex triangular-blocks whose outer cycle has length 4 and a chord.

\begin{lemma} \label{lem:outer 4-cycle}
Let $G \in \cP_n$ with $n \ge 7$, and let $B$ be a triangular-block of $G$ with no triangular holes and at most one hole. 
If $B \in \{B_{6,g}, B_{6,h}, B_{6,i}, B_{7,b}\}$, then $g(B) < 0$.
\end{lemma}
\begin{proof} We may assume that $G$ is drawn in the plane such that the outer face of $B$ is its hole. Let $C$ be the outer cycle of $B$. Let $x$ and $y$ be the ends of the unique chord of $C$, and let $a$ and $b$ be the other two vertices on the outer cycle of $B$, where $d_B(a) = 2$ if $B \ne B_{6,i}$.
If $B = B_{7,b}$, let $c$ be the interior degree 2 vertex of $B$. Note that when $|B|=6$ then $|E(B)|=11$ and $|F(B)|=7$, and when $|B|=7$ then $|E(B)|=13$ and $|F(B)|=8$.

Suppose for each petal $F$ of $B$, $|F|\ge 8$. 
Since $B$ has at least two junction vertices, if $|B| = 6$ then 
$$g(B) \le 24(6 + 4/8) - 17\cdot 11 + 6(4 + 2/2) = -1,$$
and if $|B| = 7$ then 
$$g(B) \le 24(7 + 4/8) - 17\cdot 13 + 6(5 + 2/2) = -5$$
as desired.

Thus, we may assume that there is some petal $F$ of $B$ with $|F| < 8$. 
We now show that $|F|=4$ and $F=xbytx$ for some $t\notin V(B)$. 
To see this, let $P$ be a subpath of $F$ that is also a $B$-path, and let $v$ and $w$ be the ends of $P$ in $B$.
By Lemmas \ref{lem:paths} and \ref{lem:Hpath} (applied to $B - c$ if $B = B_{7,b}$), we see that $B$ and $B-c$ have $v$-$w$ paths of any length between $d_B(v,w)$ and $5$, unless $B = B_{6,i}$ and $\{v,w\} = \{x,y\}$. Thus, since $G$ is $C_7$-free, $P=ab$, or $P=xty$ where $t\in V(F)\setminus V(B)$. 
This in particular implies  that $F$ cannot contain two such $B$-paths. 
Thus, $F \cap B$ is $xby$ or $yax$.
Then $B \ne B_{6,i}$ or else $G$ has an (18/7)-sparse set.
Hence, $|F| = 4$ and $F=xbytx$ as $a$ is a junction vertex (since $\{a\}$ is not an (18/7)-sparse set).

Note that $F$ is the unique petal of $B$ with length less than eight, $b$ is a not a junction vertex, and $xby$ is a bad cherry (thus $f_F(B)=1/3$). Therefore, 
if $B \in \{B_{6,g}, B_{6,h}\}$, we have
$$g(B) \le 24(6 + 2/8 + 1/3) - 17\cdot 11 + 6(3 + 3/2) = -2,$$ and if $B = B_{7,b}$ we have
$$g(B) \le 24(7 + 2/8 + 1/3) - 17\cdot 13 + 6(4 + 3/2) = -6,$$
as desired.
\end{proof}

To deal with other 6-vertex and 7-vertex triangular-blocks, we need a technical lemma.

\begin{lemma}\label{lem: short paths}
Let $G \in \cP_n$ with $n \ge 7$ and let $B$ be a triangular-block of $G$ with no triangular holes and $B \in \{B_{6,a}, B_{6,b}, B_{6,c}, B_{6,d}, B_{6,e}, B_{6,f},B_{7,a}\}$. Let $C$ be the outer cycle of $B$, $x_1x_2x_3$ be a subpath of $C$, and let $F_1,F_2$ be the petals of $B$ containing $x_1x_2,x_2x_3$, respectively. Suppose $G$ has a $B$-path $P$ between $x_1$ and $x_3$ of length at most 2. 
Then $|F_i|\ge 8$ for $i=1,2$, unless 
\begin{itemize}
\item $B=B_{6,a}$ with $x_1,x_3$ corresponding to $x,y$ as in Figure \ref{Hpath exceptions} and $F_1=F_2$ is a 4-cycle, or 

\item $B = B_{6,d}$ with $x_1,x_2$ corresponding to $x,y$ as in Figure \ref{Hpath exceptions} and $|F_1| \ge 8$.
\end{itemize}
\end{lemma}
\begin{proof}
First, suppose $F_1\ne F_2$.   
Then, because of $P$ we see that each $F_i$ contains an $x_{s_i}$-$x_{s_i + 1}$ $B$-path for some $s_i \in [2]$.
If $B \ne B_{6,d}$, then Lemmas \ref{lem:paths} and \ref{lem:Hpath} imply that for all $t\in \{2,3,4,5\}$, $B$ has an $x_{s_i}$-$x_{s_i+1}$ path of length $t$.
Thus, since $G$ is $C_7$-free, $|F_i|\ge 8$ for $i=1,2$. 
If $B = B_{6,d}$, then by checking cases we see that for all $t \in [4]$, $B$ has an $x_{s_i}$-$x_{s_i + 1}$ path of length $t$.
Since $|F_i| \ge 4$ (as $B$ is a triangular-block), if $|F_i| \le 6$ then by case analysis $B \cup F_i \cup P$ contains a $C_7$ unless $x_1, x_2$ correspond to $x,y$ as in Figure \ref{Hpath exceptions}, $|F_2| = 4$, and $x_2x_3x_1 \subseteq F_2$.
Hence, $|F_1| \ge 8$.

So we may assume $F_1=F_2$. Again, because of the path $P$ and planarity, $F_1\cap C=x_1x_2x_3$. So $|F_1|\ge 4$ as $B$ is a triangular-block. 

Suppose $B$ has an $x_1$-$x_3$ Hamiltonian path. Then by, Lemma~\ref{lem:paths}, $B$ has an $x_1$-$x_3$ path of length $\ell$ for every $\ell$ between $2$ and $5$. Therefore,  $|F_1|\ge 8$ to avoid a $C_7$ in $G$. 

Now assume $B$ has no $x_1$-$x_3$ Hamiltonian path.  Then by Lemma~\ref{lem:Hpath}, $B=B_{6,a}$ with $x_1,x_3$ corresponding to $x,y$ in Figure 1.    Now $B$ has an $x_1$-$x_3$ path of any length between 2 and 4. So if $|F_1|<8$ then $|F_1|=4$.
\end{proof}

We next consider $6$-vertex triangular-blocks with five vertices on the outer cycle.

\begin{lemma} \label{lem:outer 5-cycle}
Let $G \in \cP_n$ with $n \ge 7$, and let $B$ be a triangular-block of $G$ with no triangular hole.
If $B\in \{B_{6,d}, B_{6,e}, B_{6,f}\}$, then  $g(B) < 0$.
\end{lemma}
\begin{proof}
By Lemma \ref{lem:refinements}, each petal of $B$ with length at most six is equal to its refinement.
We may assume that $B$ is drawn as shown in Figure \ref{table:6-vertex blocks}.
Let $C$ be the outer cycle of $B$, so $|C| = 5$. Since no triangular face of $B$ is a hole, all junction vertices of $B$ are contained in $C$ and there are at least two (as $G$ is 2-connected) and must include all degree 2 vertices of $B$ (to avoid an (18/7)-sparse set of order 1). 

In fact, $C$ has at least three junction vertices of $B$. For, otherwise, let $S$ be obtained from $V(B)$ by removing the junction vertices of $B$ (so $|S|=4$). Then one can check that $S$ is an (18/7)-sparse set, a contradiction.

If there are no $B$-paths of length at most 2, then each petal of $B$ has length at least 8, and so 
$$g(B) \le 24(5+5/8)-7\cdot 10+(3+3/2)< 0.$$
If there are two $B$-paths of length at most 2 that have different sets of ends on $C$, then by Lemma \ref{lem: short paths}, at least three petals of $B$ have length at least 8; so $$g(B) \le 24(5 + 3/8+2/4) - 17\cdot 10 + 6(3 + 3/2) = -2.$$

So we may assume that all $B$-paths of length at most 2 in $G$ have the same ends on $B$, say $x$ and $y$, and that there is a petal $F$ of $B$ with $|F| < 8$.  
Then $F\cap B$ is a path of length 2 or 3, or else $G$ contains a $C_7$ by Lemma~\ref{lem:Hpath} and Lemma~\ref{lem:paths}.
So by Lemma~\ref{lem: short paths}, $|F|=4$ and $F$ contains a $B$-path of length 2 between $x$ and $y$, call it $P$. 
If $xy\notin E(B)$ then $B$ has an $x$-$y$ Hamiltonian path (by Lemma~\ref{lem:Hpath}), together with $P$, gives a $C_7$, a contradiction.
So $xy\in E(B)$. But then  $G$ has an (18/7)-sparse set consisting of 1 or 2 vertices, a contradiction.
\end{proof}

Finally, we consider triangular-blocks with six vertices on the outer cycle.

\begin{lemma} \label{lem:outer 6-cycle}
Let $G \in \cP_n$ with $n \ge 7$, and let $B$ be a triangular-block of $G$ with no triangular hole.
If $B \in \{B_{6,a}, B_{6,b}, B_{6,c}, B_{7,a}\}$, then $g(B) \le 0$.
\end{lemma}
\begin{proof}
By Lemma \ref{lem:refinements}, each petal of $B$ with length at most six is equal to its refinement.
We may assume that $B$ is drawn as shown in Figure \ref{table:6-vertex blocks} or Figure \ref{table:7-vertex blocks}.
Let $C$ be the outer cycle of $B$, so $|C| = 6$. 
Note that every degree 2 vertex in $B$ must be a junction vertex, to avoid an (18/7)-sparse set of order 1. Thus, if $B$ has two junction vertices then $B \in \{B_{6,a}, B_{6,b}\}$; but then $\{v\in V(B): d_B(v)\ge 3\}$ is an (18/7)-sparse set in $G$, a contradiction. So $B$ has at least three junction vertices.

If $B \ne B_{7,a}$ (so $|B|=6$, $|F(B)|=5$, and $|E(B)|=9$) and at least two edges of $C$ are in petals of $B$ of length at least 8, then  
\begin{align}
g(B) \le 24(4 + 2/8 + 4/4) - 17 \cdot 9 + 6(3 + 3/2) = 0.
\end{align}
If $B = B_{7,a}$ (so $|B|=7$, $|F(B)|=7$, and $|E(B)|=12$) and at least three edges of $C$ are in petals of $B$ of length at least 8, then 
\begin{align}
g(B) \le 24(6 + 3/8 + 3/4) - 17 \cdot 12 + 6(4 + 3/2) = 0.
\end{align}

Thus, we may assume that at most one edge (if $|B| = 6$) or two edges (if $|B| = 7$) of $C$ are in petals of $B$ of length at least 8.
Let $\cP$ be the collection of $B$-paths of $G$ with at most two edges.
Lemmas \ref{lem:paths} and \ref{lem:Hpath} imply that for each petal $F$ of $B$ with $|F| < 8$, $F$ contains a path in $\cP$.
We consider three cases.

First suppose that $B = B_{6,a}$.
Let $v$ be a degree 2 vertex of $B$.
Suppose that $v$ is in a petal $F$ with $|F| < 8$.
Since $G$ has no degree 2 vertices, $F$ contains a path $P$ in $\cP$ with $v$ as an end.
Since $v$ is not the end of a chord of $C$ and $G$ contains no $C_7$, it follows from Lemma \ref{lem:Hpath} that $P$ consists of a single edge.
By Lemma \ref{lem: short paths} and (1), the other end of $P$ is at distance three from $v$ on $C$.
But then applying the same reasoning to the other degree 2 vertex $u$ of $B$ shows that at least two edges of $C$ are each contained in a petal of length at least 8, and so (1) implies that $g(B) \le 0$.

Next suppose that $B \in \{B_{6,b}, B_{6,c}\}$.
Then by Lemma \ref{lem: short paths} and (1) we may assume that the ends of each $P \in \cP$ are at distance three on $C$.
This implies that each petal of $B$ with $|F| < 8$ contains exactly one path in $\cP$ as a subpath.
It is straightforward to check that $G$ has an (18/7)-sparse set unless $B = B_{6,b}$ and the ends of each $P \in \cP$ are the degree 2 vertices of $B$.
Since the degree 2 vertices of $B$ are not the ends of a chord of $C$, it follows that each $P \in \cP$ consists of a single edge to avoid a $C_7$ in $B \cup P$, and so $|\cP| = 1$.
Therefore there is a unique petal $F$ of $B$ with $|F| < 8$.
Since only three edges of $C$ are contained in $F$, there are three edges of $C$ contained in petals of length at least 8.
Then (1) implies that $g(B)  \le 0$.

Finally, suppose that $B = B_{7,a}$.
If there is a path $P$ in $\cP$ with both ends having degree 2 in $B$ then $B \cup P$ contains $C_7$, a contradiction.
By Lemma \ref{lem: short paths} and (2) there is at most one path in $\cP$ with neither end having degree 2 in $B$, and clearly there is at most one path in $\cP$ with exactly one end having degree 2 in $B$.
So $\cP$ contains at most two paths and, since $G$ has no degree 2 vertex, it follows that there is a unique petal of $B$ with $|F| < 8$ (and $F$ contains both paths in $\cP$).
Since $F$ contains at most three edges of $C$, at least three edges of $C$ are each contained in a petal of $B$ of length at least 8; so  (2) implies that $g(B) \le 0$.
\end{proof}


\section{Five-vertex triangular-blocks} \label{sec:5-vertex blocks}

In this section we analyze the possible $5$-vertex triangular-blocks of a graph $G \in \cP_n$.

\begin{lemma}
Let $G \in \cP_n$ with $n \ge 7$, and let $B$ be a triangular-block of $G$,  and assume that $G$ is drawn in the plane such that $B = B_{5,a}$ is drawn as in Figure \ref{table:5-vertex blocks}. Then $g(B) \le 0$.
\end{lemma}
\begin{proof}
By Lemma \ref{lem:refinements}, each petal of $B$ with length at most six is equal to its refinement.
Lemma \ref{lem:interior path} shows that no triangular face of $B$ is a hole.
Let $C=abcdea$ be the outer cycle of $B$ with $d_B(a)=4$. Note that $B$ has at least three junction vertices, or else $\{b\}$ or $\{e\}$ or $\{c,d\}$ would be an (18/7)-sparse set in $G$. 

If at least two edges of $C$ are contained in petals of $B$ of length at least 8, then
$$g(B) \le 24(3 + 2/8 + 3/4) - 17\cdot 7 + 6(2 + 3/2) = -2.$$
So we may assume that at most one edge of $C$ is contained in a petal of $B$ of length at least 8.
This implies, up to relabeling vertices, that the edges $ab$ and $bc$ are not in a petal of length at least 8.
Let $F_1$ and $F_2$ be petals of $B$ that contain $ab$ and $bc$, respectively.

Suppose $F_1$ and $F_2$ are both leaky.
It is straightforward to check that there is a pair $(P_1, P_2)$ of length 2 $B$-paths where $P_i \subseteq F_i$, the ends of $P_i$ are not $\{a,c\}$ or $\{a,d\}$, and $P_1$ and $P_2$ have distinct sets of ends.
But then $B \cup P_1 \cup P_2$ contains $C_7$, a contradiction.
So either $F_1$ or $F_2$ is not leaky.

We claim that if $F_1$ is not leaky then $F_1\cap C=eab$ and if $F_2$ is not leaky then $F_2\cap C=bcd$. Suppose $F_1$ is not leaky. First, $F_1\cap C$ does not contain $deab$ or $abc$, as otherwise $\{e\}$ or $\{b\}$ would be an (18/7)-sparse set in $G$. Now suppose $F_1\cap C=ab$. Then since $B$ has an $a$-$b$ path of every length between 1 and $4$ and $|F|<8$, $G$ has a $C_7$, a contradiction.
Therefore $F_1\cap C=eab$, and similar argument shows that if $F_2$ is not leaky then $F_2\cap C=bcd$.

Suppose $F_1$ and $F_2$ are both non-leaky. Then by the above claim, $|F_1|=|F_2|=4$ (to avoid $C_7$ in $G$). Let $F_1=veabv$ and $F_2=wbcdw$. If $v\ne w$, we see that $B\cup F_1\cup F_2$ contains a $C_7$, a contradiction. So $v=w$. But then $\{b,c\}$ is an (18/7)-sparse set in $G$, a contradiction. 

 We may thus assume that $F_1$ is not leaky and $F_2$ is leaky, as the same argument applies to the case when $F_1$ is leaky and $F_2$ is not leaky. Hence, $F_1\cap C=eab$ and $|F_1|=4$.

   We claim that $F_2\cap C$ consists of $bc$ and $e$. For, otherwise, $F_2\cap C$ contains a $B$-path $P$ between $c$ and $d$ which is also disjoint from $F_1$ (by planarity). Note that $P$ has length between 2 and 5 and that $B\cup F_1$ has $c$-$d$ paths of every length between 1 and 5. Hence, $B\cup F_1\cup P$ contains $C_7$, a contradiction. 

   The same argument in the above paragraph shows that $ce$ is an edge of $F_2$. Moreover, the $b$-$e$ subpath of $F_2-c$ has length 2, to avoid forming a $C_7$ with a $b$-$e$ path in $B\cup F_1$. 
   Let $F_1=veabv$ and $F_2=wbcew$. If $v=w$ then $d_G(v)=2$ and $\{v\}$ is an (18/7)-sparse set in $G$, a contradiction. So $v\ne w$. 

 Now the petals of $B$ containing $cd$ and $de$ are distinct. For, otherwise, the petal containing $cde$ has length between 4 and 6, which, together with a $c$-$e$ path in $B\cup P$ forms a $C_7$, a contradiction. Moreover, the petals containing $cd$ or $de$ each contains  a $(B\cup F_1)$-path of length at least 3 with both ends in $\{c,d,e\}$. But then one of these paths has length at most 6, which forms a $C_7$ with a path in $B$, a contradiction. 
\end{proof}

\begin{lemma}
Let $G \in \cP_n$ with $n \ge 7$ and let $B$ be a triangular-block of $G$.  Suppose that $G$ is drawn in the plane such that $B = B_{5,d}$ is drawn as in Figure \ref{table:5-vertex blocks} and its only hole is its outer face. Then $g(B) \le 0$.
\end{lemma}
\begin{proof}
 First, we observe that every $B$-path in $G$ has length $\ell$ with $\ell=2$ or $\ell \ge 7$. For, otherwise, suppose $P$ is a $B$-path in $G$ of length $\ell$ with $3\le \ell\le 6$. Note that $B$ contains a path $Q$ of length $7-\ell$ between the ends of $P$. Now $P\cup Q$ is a $C_7$ in $G$, a contradiction. 

Let $uvwu$ be the outer cycle of $B$. 
(Note that we do not specify the vertices $u,v,w$.)
Thus, all junction vertices of $B$ are contained in $\{u,v,w\}$ and, since $G$ is 2-connected, at least two of $\{u,v,w\}$ are junction vertices of $B$.  Hence, if each edge of $uvwu$ is in a petal of $B$ of length at least 8, then 
$$g(B) \le 24(5 + 3/8) - 17\cdot 9 + 6(3 + 2/2) = 0,$$
so we may assume that there is a petal $F$ of $B$ with $|F| < 8$. 

We claim that $|F|=4$ and $F$ has a bad cherry (so $f_F(B)=1/3$). To see this, let $P$ be a longest $B$-path contained in $F$. By the observation above, $P$ has length 2 (as $|F|<8$). Since $B$ is a triangular-block, $P$ and two edges of $uvwu$ bounds the face $F$.   So $|F| = 4$, and $F\cap B$ is a bad cherry. Without loss of generality, assume that $F\cap B=vwu$.

We next consider the petal $F_1$ of $B$ containing $uv$.
Note that $F_1-uv$ is a $B$-path, because $w$ is in the interior of $F$.
Hence, since $|F_1|\ge 4$, $F_1-uv$ has length at least 7 (by the above observation). So $|F_1| \ge 8$ (and $f_{F_1}(B)\le 1/8$).

We now show that $u$ and $v$ are each contained in at least three triangular-blocks of $G$. For, suppose, without loss of generality, that $v$ is in exactly two triangular-blocks of $G$. Let $x$ be the vertex in $F-\{u,v,w\}$. 
If $N_G(v)\subseteq V(B\cup F)$,  then $V(B)\setminus \{u\}$ is an (18/7)-sparse set in $G$, a contradiction. So $v$ has a neighbor outside $B\cup F$, say $y$.
Then $vy$ and $vx$ are in some triangular-block $B'$ of $G$. 
One can check that $B'$ or $G[V(B')\cup \{u\}]$ contains a $v$-$u$ path of length of 3 or 4, which together with a $u$-$v$ path in $B$ forms a $C_7$ in $G$, a contradiction.  
Therefore, 
$$g(B) \le 24(5 + 1/8 + 1/3) - 17\cdot 9 + 6(3 + 2/3) = 0,$$
as desired.
\end{proof}

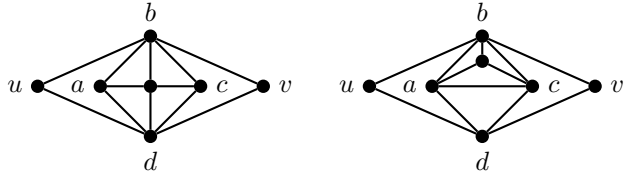
\begin{figure} \label{exceptions}
\begin{center}
\begin{tabular}{cc}
    \begin{tikzpicture}
     \node[label = right: {$c$}] (1) at (0:2\R/3) {};
     \node[label = above: {$b$}] (2) at (90:2\R/3) {};
     \node[label = left: {$a$}] (3) at (180:2\R/3) {};
     \node[label = below: {$d$}] (4) at (270:2\R/3) {};
     \node[label = right: {$v$}] (5) at (0:1.5\R) {};
     \node[label = left: {$u$}] (6) at (180:1.5\R) {};
     \node (7) at (0:0) {};
     \draw (1) -- (2);
\draw (2) -- (3);
\draw (3) -- (4);
\draw (1) -- (4);
\draw (2) -- (5);
\draw (4) -- (5);
\draw (2) -- (6);
\draw (4) -- (6);
\draw (1) -- (7);
\draw (2) -- (7);
\draw (3) -- (7);
\draw (4) -- (7);
\end{tikzpicture}  
&
\begin{tikzpicture}
     \node[label = right: {$c$}] (1) at (0:2\R/3) {};
     \node[label = above: {$b$}] (2) at (90:2\R/3) {};
     \node[label = left: {$a$}] (3) at (180:2\R/3) {};
     \node[label = below: {$d$}] (4) at (270:2\R/3) {};
     \node[label = right: {$v$}] (5) at (0:1.5\R) {};
     \node[label = left: {$u$}] (6) at (180:1.5\R) {};
     \node (7) at (90:\R/3) {};
     \draw (1) -- (2);
\draw (2) -- (3);
\draw (3) -- (4);
\draw (1) -- (4);
\draw (2) -- (5);
\draw (4) -- (5);
\draw (2) -- (6);
\draw (4) -- (6);
\draw (1) -- (3);
\draw (1) -- (7);
\draw (2) -- (7);
\draw (3) -- (7);
\end{tikzpicture}
\end{tabular}
\end{center}
\caption{The two exceptional flowers.}
\end{figure}

\begin{lemma}
Let $G \in \cP_n$ with $n \ge 7$ and let $B$ be a triangular-block of $G$. Suppose that $G$ is drawn in the plane such that $B \in \{ B_{5,b}, B_{5,c}\}$ is drawn as in Figure \ref{table:5-vertex blocks}. Then $g(B) \le 0$ unless the flower of $B$ is one of the two graphs shown in Figure \ref{exceptions}, up to redrawings that preserve facial cycles.
\end{lemma}
\begin{proof}
By Lemma \ref{lem:refinements}, each petal of $B$ with length at most six is equal to its refinement.
Lemma \ref{lem:interior path} shows that the only hole of $B$ is its outer face.
Let $C=abcda$ be the outer cycle of $B$, with $d_B(d)=2$ if $B=B_{5,c}$. Note that no petal of $B$ has a bad cherry, as otherwise it is easy to see that $B=B_{5,c}$ and one or two vertices of $B$ would form an (18/7)-sparse set in $G$.
Since $G$ is 2-connected, $B$ has at least 2 junction vertices. Thus, if each edge of $C$ is in a petal of $B$ of length at least 8, then 
$$g(B) \le 24(4 + 4/8) - 17 \cdot 8 + 6(3 + 2/2) = -4.$$
So we may assume that there is a petal $F$ of $B$ with $|F| < 8$.

Thus, $F$ contains a $B$-path $P$ of length between 2 and 5. To avoid forming $C_7$ with a path in $B$, $P$ has length 2 if the ends of $P$ do not induce a chord of $C$, and  $P$ has length 2 or 3 otherwise.   

Suppose that we may choose $P$ so that the ends of $P$ are adjacent on $C$. 
Then $P$ has length 2 to avoid forming a $C_7$ with $B$, and, since $B$ is a triangular-block, $F$ is a leaky petal.  Therefore, $F$ contains a second $B$-path $P'$ which does not have the same two ends as $P$, and hence $B$ has at least three junction vertices. Note that $B\cup P$ is a near triangulation on 6 vertices, and also note by the above paragraph that $P'$ has length at most 3. 
It is straightforward to check that $P'$ has length 1, or else $B \cup P'$ or $B \cup P \cup P'$ would contain $C_7$. Thus, if $B=B_{5,c}$ then $P'\ne ac$. Hence, without loss of generality we may assume that $P'=bd$, and that the triangle $abda$ separates $B$ from some vertex of $G$.
Then the petal of $B$ containing $ab$ either contains an $a$-$b$ path of length at least 7 or a $b$-$d$ path of length at least 6 to avoid forming a $C_7$ with $B$, and in either case has length at least 8.
Similarly, the petal of $B$ containing $ad$ has length at least 8.
Therefore
$$g(B) \le 24(4 + 2/8 + 2/4) - 17\cdot 8 + 6(2 + 3/2) = -1,$$
as desired.

So we may assume that $F$ contains no $B$-path with length between 2 and 5 whose ends are adjacent on $C$.
Then $F$ is not a leaky petal.
If $B = B_{5,c}$ and $P$ is an $a$-$c$ path, then $G$ has a sparse set consisting of 1 or 2 vertices from $V(B)\setminus \{a,c\}$, a contradiction.
So we may assume without loss of generality that $P$ is a $b$-$d$ path (so $P$ has length 2 to avoid forming a $C_7$ with $B$) and that $F$ is the cycle $bcdvb$, where $v$ is the internal vertex of $P$.

If $B$ has three junction vertices, then the petal $F_1$ of $B$ containing $ab$ also contains a $B$-path of length at least 2 with ends $a$ and $b$ or $a$ and $d$.
Since such a path has length at least 7 (to avoid forming a $C_7$ with $B$), $|F_1| \ge 8$.
Similarly, the petal of $B$ containing $ad$ has length at least 8.
Then
$$g(B) \le 24(4 + 2/8 + 2/4) - 17\cdot 8 + 6(2 + 3/2) = -1.$$

So we may assume that $B$ has exactly two junction vertices, namely $b$ and $d$.
Then the petal $F_1$ of $B$ containing the path $dab$ is the union of $dab$ and a $B$-path $Q$ in $G$ between $d$ and $b$. As with $P$ and $F$,
$Q$ has length 2 or at least 6 to avoid creating $C_7$ with $B$.  
If $Q$ has length 2, then $v \notin V(Q)$ and the flower of $B$ is shown in Figure \ref{exceptions}.

So assume $Q$ has length at least 6. Then $|F_1| \ge 8$. 
If $b$ and $d$ are each in at least three triangular-blocks of $G$ then 
$$g(B) \le 24(4 + 2/8 + 2/4) - 17\cdot 8 + 6(3 + 2/3) = 0.$$
So assume by symmetry that $b$ is in exactly two triangular-blocks of $G$.
If $N_G(b)\subseteq V(B\cup P)$ then $V(B-d)$ is an (18/7)-sparse set in $G$, a contradiction.
So $b$ has a neighbor $z$ outside $B\cup P$. 
Then $bv$ and $bz$ are in the same triangular-block $B'$ of $G$. Now $B'$ contains a $b$-$v$ path of length between 2 and 6, which is also a $(B \cup P)$-path. However, this path and a $b$-$v$ path in $B\cup P$ form a $C_7$ in $G$, a contradiction.
\end{proof}

Finally, for the triangular-blocks with flower shown in Figure \ref{exceptions}, we note that each edge of its outer cycle is contained in a facial cycle of length at least 5, to avoid a $C_7$ in $G$. Thus, we have the following. 

\begin{lemma} \label{lem:exceptional blocks}
Let $G \in \cP_n$ with $n \ge 7$.
Let $B$ be a triangular-block of $G$ with flower shown in Figure \ref{exceptions}, and let $e \in E(G) \setminus  E(B)$ be in a petal of $B$.
Then $e$ is incident with exactly one face of $G$ of length less than five.
In particular, $e$ is a trivial triangular-block.
\end{lemma}


\section{Small Triangular-Blocks} \label{sec:small blocks}

In this section we analyze the triangular-blocks with at most four vertices in a graph $G \in \cP_n$.
We first consider trivial triangular-blocks.

\begin{lemma}
Let $G \in \cP_n$ with $n \ge 7$ and let $B$ be trivial triangular-block of $G$. Then $g(B) \le 0$.
\end{lemma}
\begin{proof}
Let $x_1$ and $x_2$ be the two vertices of $B$, and let $F_1$ and $F_2$ be the two petals of $B$.
Note that $|F_i| \ge 4$ for each $i \in [2]$ since $B$ is a triangular-block.
Each of $x_1$ and $x_2$ is in at least two triangular-blocks, so $n(B) \le \frac{1}{2} + \frac{1}{2} = 1$.

We may assume that for each $i\in [2]$, $|F_i|=4$, $F_i$ equals its refinement, and $x_i$ is contained in exactly two triangular-blocks of $G$ (including $B$). First, if the refinement of $F_1$ is not $F_1$ then Lemma \ref{lem:refinements} implies that $|F_1|\ge 8$ and $|F_2|\ge 4$, or $|F_1| = 4$ and $|F_2| \ge 8$; so 
$$g(B) \le 24(1/8 + 1/4) - 17 + 6(1/2+1/2) = -2.$$
So we may assume that $F_1$ and $F_2$ are both equal to their refinements.
Next, if $|F_1| \ge 5$ or $|F_2|\ge 5$ then 
$$g(B) \le 24(1/4 + 1/5) - 17 + 6(1/2 + 1/2) = -1/5,$$ so we may assume that $|F_1| = |F_2| = 4$. 
Now assume $x_1$ or $x_2$ is in at least three triangular-blocks of $G$. Then $n(B) \le \frac{1}{2} + \frac{1}{3}$ and again 
$$g(B) \le  24(1/4 + 1/4) - 17 + 6(1/2 + 1/3) = 0.$$

So both $x_1$ and $x_2$ are each contained in exactly 2 triangular-blocks.
Let $x_3$ and $x_4$ be the vertices of $F_1$ other than $x_1$ and $x_2$, such that $F_1=x_1x_2x_3x_4x_1$. 
\medskip

{\it Case} 1. $V(F_1\cap F_2)=\{x_1,x_2\}$. 

Let $x_5$ and $x_6$ be the other two vertices of $F_2$, such that $F_2=x_1x_2x_5x_6x_1$. Now $d_G(x_i)\ge 4$ for some $i\in [2]$, as $\{x_1,x_2\}$ is not an (18/7)-sparse set. By symmetry, assume $d_G(x_2)\ge 4$. Consider the triangular-block $B'$ of $G$ such that $B'$ contains $x_3x_2x_5$, and let $F_3$ and $F_5$ be the facial triangles of $B'$ containing $x_2x_3$ and $x_2x_5$, respectively.
If $F_3$ contains a third vertex in $\{x_1,x_2,x_3,x_4,x_5,x_6\}$, it must be $x_6$ to avoid an (18/7)-sparse set of order 1. 
But then $F_5$ contains a vertex $v$ with $v\notin \{x_1,x_2,x_3,x_4,x_5,x_6\}$, and
$x_2vx_5x_6x_1x_4x_3x_2$ is a $C_7$ in $G$, a contradiction.

\medskip

{\it Case} 2. $V(F_1\cap F_2)\ne \{x_1,x_2\}$ 

Then $|V(F_1\cap F_2)|=3$.
Without loss of generality we may assume that $x_4\in V(F_2)$.
Let $x_5$ be the fourth vertex of $F_2$; then $F_2 = x_1x_2x_4x_5x_1$.
Note that $x_4x_1x_5x_4$ bounds a triangular-block $B_1$ of $G$ since $x_1$ is in exactly 2 triangular-blocks of $G$. 
Since $F_2$ is a facial cycle and $\{x_5\}$ cannot be a sparse set in $G$, $|B_1|\ge 4$.
Since $x_1x_4$ in an edge on the outer cycle of $B_1$, Lemmas \ref{lem:Hpath} and \ref{lem:paths} imply that $B_1$ contains an $x_1$-$x_4$ path $P$ of length 3. 
Note that the path $x_3x_2x_4$ is contained in a triangular-block $B_2$ of $G$, and the facial 
triangle of $B$ containing $x_2x_3$ 
contains a vertex $v\notin \{x_1,x_2,x_3,x_4,x_5\}$ as $\{x_3\}$ cannot be an (18/7)-sparse set of $G$. But then $P\cup x_1x_2vx_3x_4$ is a $C_7$ in $G$, a contradiction.
\end{proof}

\begin{lemma}
Let $G \in \cP_n$ with $n \ge 7$ and let $B$ be a triangular-block of $G$.
If $B = B_3$ as in Figure \ref{table:small blocks}, then $g(B) \le 0$.
\end{lemma}
\begin{proof}
By Lemma \ref{lem:refinements} and to avoid a $C_7$, each petal of $B$ with length at most six is equal to its refinement.
Note that each vertex of $B$ is a junction vertex, since $G$ has no degree 2 vertex.
Hence, 
$$g(B) \le 24(1 + 3/4) - 17\cdot 3 + 6(3/2) = 0,$$
as desired.
\end{proof}

\begin{lemma}
Let $G \in \cP_n$ with $n \ge 7$ and let $B$ be a triangular-block of $G$. Suppose $G$ is drawn in the plane such that $B=B_{4,b}$ is drawn as in Figure \ref{table:small blocks} with its outer face as its only hole. Then $g(B) \le 0$.
\end{lemma}
\begin{proof}
By Lemma \ref{lem:refinements}, each petal of $B$ with length at most six is equal to its refinement.
Let $a,b,c$ be the vertices of the outer cycle of $B$, which are all junction vertices (to avoid an (18/7)-sparse set consisting of two adjacent vertices).
If there is a petal of $B$ with length at least 8, then 
$$g(B) \le 24(3 + 1/8 + 2/4) - 17\cdot 6 + 6(1 + 3/2) = 0.$$
So we may assume that each petal of $B$ has length less than 8. 

Suppose $B$ has a leaky petal, say $F_1$, and assume by symmetry that  $bc\in E(F_1)$.
Then $F_1$ is the union of $bc$, a $B$-path $P_1$ between $a$ and $b$, and a $B$-path $P_2$ between $a$ and $c$. 
Note that for each $i\in [2]$, $P_i$ has length 2 (or else $B \cup P_1 \cup P_2$ would contain $C_7$), and let $v_i$ be the internal vertex of $P_i$.
Now, the petal $F_2$ of $B$ containing $ac$ 
contains a $B$-path $Q$ between $a$ and $c$. 
Note that the length of $Q$ is between $3$ and $5$, and hence  $Q\cup P_1\cup B$ contains a $C_7$, a contradiction.  

So we may assume that $B$ has no leaky petal. 
Since $a,b,c$ are all junction vertices of $G$ and there is no $B$-path of length 4 or 5 in $G$ (to avoid forming a $C_7$ with $B$), each petal of $B$ has length 4.
Let $F_1,F_2,F_3$ be the petals containing $ac,ab,bc$, respectively. 

Let $x$ and $y$ be the vertices of $F_1$ that are not in $B$ so that $F_1=axyca$.
Then $F_2$ contains a vertex $z$ that is not in $B \cup F_1$ (to avoid an (18/7)-sparse set of order at most 2 in $G$) and such $z$ is unique to avoid a $C_7$ in $B\cup F_1\cup F_2$. So $z$ has two neighbors on $B \cup F_1$ that are in $F_2$: $a$ and $c$, or $a$ and $y$,  or $b$ and $x$. If these neighbors of $z$  are $b$ and $x$ then $B\cup F_1\cup F_2$ contains $C_7$, a contradiction. 
If the neighbors are $a$ and $c$ then $b$ is not a junction vertex, a contradiction.
So the neighbors of $z$ are $a$ and $y$, and then the fourth edge of $F_2$ is $by$ or $ay$.

By applying the same reasoning to the petals $F_1$ and $F_3$, we deduce that $bx\in E(F_3)$ or $cx \in E(F_3)$, and this is a contradiction because $b$ and $x$ are separated by the cycle $azyca$ or $ayca$.
\end{proof}

\begin{lemma}
Let $G \in \cP_n$ with $n \ge 7$ and let $B$ be a triangular-block of $G$. Suppose $G$ is drawn in the plane such that $B=B_{4,a}$ is drawn as in Figure \ref{table:small blocks} with its outer face as its only hole. Then $g(B) \le 0$.
\end{lemma}
\begin{proof}
By Lemma \ref{lem:refinements}, each petal of $B$ with length at most six is equal to its refinement.
Lemma \ref{lem:interior path} shows that $B$ has no triangular hole.
Let $C=abcda$ be the outer cycle of  $B$, with $a,c$ of degree 3 in $B$. Note that $b,d$ must be junction vertices as $\{b\}$ and $\{d\}$ cannot be an (18/7)-sparse set. 
If both $a$ and $c$ are junction vertices, then 
$$g(B) \le 24(2 + 4/4) - 17\cdot 5 + 6(4/2) = -1.$$
So we may assume by symmetry that $c$ is not a junction vertex. Then $a$ is a junction vertex (as $\{a,c\}$ is not an (18/7)-sparse set), and the path $bcd$ is contained a petal of $B$, say $F_1$. 

We may assume  $|F_1|=4$, or else, $|F_1|\ge 5$ and 
$$g(B) \le 24(2 + 2/4 + 2/5) - 17\cdot 5 + 6(1 + 3/2) = -2/5.$$
Then $F_1$ is not leaky. 
Let $F_1=bcdvb$, where $v\notin V(B)$.

Let $F_2,F_3$ denote the petals of $B$ containing $ab,ad$, respectively.
Then $F_2\ne F_3$, or else $\{a,c\}$ would be an (18/7)-sparse set in $G$.  
We may assume $|F_2| = 4$; for otherwise, $|F_2|\ge 5$ and $|F_3|\ge 5$, and $$ g(B) \le 24(2 + 2/4 + 2/5) - 17\cdot 5 + 6(1 + 3/2) = -2/5.$$

Note that $F_2$ contains a unique vertex, say $w$, not contained in $B \cup F_1$, to avoid a $C_7$ in $B\cup F_1\cup F_2$. Let $e_1,e_2$ be the edges of $F_2-ab$ incident with $a,b$, respectively. Then $e_1, e_2 \notin E(B\cup F_1)$, or else  $\{a,c\}$ or $\{b,c\}$
would be an (18/7)-sparse set in $G$. 
Since $e_1$ and $e_2$ cannot both be incident to $w$, 
 either $e_1 = av$ and $e_2=bw$, or $e_1=aw$ and $e_2 = bd$.

If $e_1=aw$ and $e_2 = bd$ then $w$ is incident to $a$ and $d$.
But then  $F_3$  consists of $ad$ and a $(B\cup F_1)$-path between $a$ and $d$, which has length at least 3 (as $B$ is a triangular-block) and forms a $C_7$ with $B\cup F_1$, a contradiction. 

So  $e_1 = av$ and $e_2=bw$, which implies that  $F_2= abwva$.
Since $\{c,d\}$ is not an (18/7)-sparse set in $G$, $F_3$ contains a  $(B \cup F_1 \cup F_2)$-path between $a$ and $d$ or between $d$ and $v$. 
But then $|F_3| \ge 8$ to avoid a $C_7$ in $B \cup F_1 \cup F_3$ or $B \cup F_1 \cup F_2 \cup F_3$, and so 
$$g(B) \le 24(2 + 3/4 + 1/8) - 17\cdot 5 + 6(1 + 3/2) = -1,$$
as desired.
\end{proof}


\section{The $2$-Connected Case} \label{sec:2-connected case}

The lemmas in the previous sections combine to show the following.

\begin{lemma} \label{lem:hole-free blocks}
Let $G \in \cP_n$ with $n \ge 7$, and let $B$ be a triangular-block of $G$ with only one hole, and no triangular hole unless $B$ is a triangulation. Then $g(B) \le 0$, with the exception of the two cases shown in Figure \ref{exceptions}.
\end{lemma}

We next show that this is true even for triangular-blocks with more than one hole.

\begin{lemma} \label{lem:counting blocks}
Let $G \in \cP_n$ with $n \ge 7$. If $B$ is a triangular-block of $G$ then $g(B) \le 0$, with the exception of the two cases shown in Figure \ref{exceptions}.
\end{lemma}
\begin{proof}
We draw $G$ so that $B$ is drawn as shown in Figures \ref{table:small blocks} - \ref{table:7-vertex blocks} and the outer face of $B$ is a hole. If $B$ has no other holes, then the statement follows from Lemma \ref{lem:hole-free blocks}. So 
let $\cF$ be the set of facial cycles of $B$ bounding a hole other than the outer face of $B$. Note that each $F \in \cF$ satisfies $|F| \le 4$.

Let $G'$ be obtained from $G$ by deleting all vertices and edges in the interior of $F$ for all $F\in \cF$. 
Let $e'(B)$, $n'(B)$, $f'(B)$, and $g'(B)$ be the values of $e(B)$, $n(B)$, $f(B)$, and $g(B)$ calculated with respect to $G'$. By Lemma \ref{lem:hole-free blocks}, we have $g'(B)\le 0$. 
Note that $e(B) = e'(B)$, and $n(B) \le n'(B)$ because for each vertex $v$ of $B$, the number of triangular-blocks in $G$ containing $v$ is at least the number of triangular-blocks in $G'$ containing $v$.
Also, each $F \in \cF$ contributes $1$ towards $f'(B)$ and at most $4/4$ towards $f(B)$, because $|F|\le 4$,  
and each petal of $B$ has length at least four since $B$ is a triangular-block. So $f(B) \le f'(B)$.
Therefore, $g(B) \le g'(B) \le 0$.
\end{proof}

We can now combine Lemma \ref{lem:partition} with Lemma \ref{lem:counting blocks} to show that Theorem \ref{main} holds for graphs in $\cP_n$ with $n \ge 7$.

\begin{theorem}
Let $G \in \cP_n$ with $n \ge 7$.
Then $e(G) \le \frac{18n}{7} - \frac{48}{7}$.
\end{theorem}
\begin{proof}
Let $\cB$ be the collection of triangular-blocks of $G$, and let $\cB_1 \subseteq \cB$ be the collection of all exceptional triangular-blocks of $G$ with flower shown in Figure \ref{exceptions}.

We now define a suitable partition $\cP$ of $\cB$.
Let $B \in \cB$.
If $B \in \cB_1$, then (by Lemma \ref{lem:exceptional blocks}) let $e_1, e_2, e_3, e_4$ be the four trivial blocks in $G$ that are contained in the flower of $B$, and let $\{B, e_1, e_2, e_3, e_4\} \in \cP$.
If $B \notin \cB_1$ and is not a trivial block in $G$ contained in the flower of a triangular-block in $\cB_1$, let $\{B\} \in \cP$.
To show that $\cP$ is a partition of $\cB$, it suffices to show that there is no trivial triangular-block $e$ that is in the flowers of two blocks in $\cB_1$.
This follows from Lemma \ref{lem:exceptional blocks}, because $e$ would be incident with two faces of length four.

By Lemma \ref{lem:partition}, it suffices to show that each set $\cB' \in \cP$ satisfies $\sum_{B \in \cB'} g(B) \le 0$.
If $\cB'=\{B\}$ for some triangular-block $B$ in $G$, then $B$ is not exceptional and it follows from Lemma \ref{lem:counting blocks} that $\sum_{B \in \cB'} g(B) =g(B) \le 0$.
Otherwise $\cB'=\{B,e_1,e_2,e_3,e_4\}$ for some exceptional block $B$ and the four trivial blocks $e_1,e_2,e_3,e_4$ in the flower of $B$.
For each $i \in [4]$, the edge $e_i$ is incident with a face of length at least five by Lemma \ref{lem:exceptional blocks}, and both its ends are in at least three triangular-blocks, and so 
$$g(e_i) \le 24(1/4 + 1/5) - 17 + 6(2/3) = -11/5.$$
Since 
$$g(B) \le 24(4 + 4/4) - 17\cdot 8 + 6(3 + 2/2) = 8$$
using Lemma \ref{lem:refinements},
we calculate that $\sum_{B \in \cB'} g(B) \le 8 - 4(11/5) = -4/5 < 0$.
Therefore each $\cB' \in \cP$ satisfies $\sum_{B \in \cB'} g(B) \le 0$, and Lemma \ref{lem:partition} implies that $e(G) \le \frac{18n}{7} - \frac{48}{7}$.
\end{proof}


\section{The Main Proof} \label{sec:main proof}

We now obtain Theorem \ref{main} from Theorem \ref{main good} by considering small separations and sparse sets.

\begin{proof}[Proof of Theorem \ref{main}]
Let $\cG_0$ be the class consisting of all planar graphs with at most six vertices or with a planar drawing in $\cP_n$ for some $n \ge 7$.
For each positive integer $i$, let $\cG_i$ be the union of $\cG_{i-1}$ and the class of $C_7$-free planar graphs obtained from any of the following operations:
\begin{enumerate}[(1)]

\item Add a set $S$ of vertices, and edges with at least one vertex in $S$ to a graph in $\cG_{i-1}$ such that $|S|\le 4$ and $S$ is an (18/7)-sparse set in the new graph. 

\item Take the disjoint union of two graphs in $\cG_{i-1}$.

\item Glue two graphs in $\cG_{i-1}$ together at a vertex.
\end{enumerate}
Let $\cG = \cup_{i\ge 0} \cG_i$, and note that $\cG$ is precisely the class of $C_7$-free planar graphs.
We will show that operations (1)-(3) can only decrease density.

\begin{claim} \label{claim:(1)-(4)}
Let $G_1$ be an $n_1$-vertex $C_7$-free planar graph with $e(G_1) \le \frac{18n_1}{7} - c_1$ for some constant $c_1$.
If $G$ is an $n$-vertex $C_7$-free planar graph obtained from $G_1$ by operation $(1)$ by adding an (18/7)-sparse set $S$, then $e(G) \le \frac{18n}{7} - c_1 -\left(\frac{18|S|}{7}-\lfloor \frac{18|S|}{7}\rfloor\right)$.
\end{claim}
\begin{proof} Note that $n=n_1+|S|$ and 
$$e(G) \le \frac{18n_1}{7} - c_1 +\Bigl\lfloor \frac{18|S|}{7}\Bigr\rfloor= \frac{18n}{7} - c_1 - \left(\frac{18|S|}{7}-\Bigl\lfloor \frac{18|S|}{7}\Bigr\rfloor\right).$$
\end{proof}

Note that when $|S|\in [2]$ this gives a saving of $1/7$, and when $|S|\in \{3,4\}$ this gives a saving of $2/7$. We perform a similar calculation for operations (2) and (3).

\begin{claim} \label{claim:(5) or (6)}
For each $i \in[2]$, let $G_i$ be an $n_i$-vertex $C_7$-free planar graph with $e(G_i) \le \frac{18n_i}{7} - c_i$ for a constant $c_i$.
If $G$ is an $n$-vertex $C_7$-free planar graph obtained by applying operation $(2)$ or $(3)$ to $G_1$ and $G_2$, then
$e(G) \le \frac{18n}{7} - c_1 - c_2 + \frac{18}{7}$.
Moreover, if $(3)$ applies, then $e(G) \le \frac{18n}{7} - c_1 - c_2$.
\end{claim}
\begin{proof}
We have 
\begin{align*}
e(G) &= e(G_1) + e(G_2) \\
&\le \frac{18n_1}{7} - c_1 + \frac{18n_2}{7} - c_2 \\
&\le  \frac{18(n + 1)}{7} - c_1 - c_2 \\
& = \frac{18n}{7} - c_1 - c_2 + \frac{18}{7}.
\end{align*}
If (3) applies, then $n_1 + n_2 = n$ and so $e(G) \le \frac{18n}{7} - c_1 - c_2$.
\end{proof}

Let $\cH_0$ be the class of planar graphs with at least 3 and at most 6 vertices, and for each positive integer $i$, let $\cH_i$ be the union of $\cH_{i-1}$ and the class of $C_7$-free planar graphs obtained from applying any of operations (1) - (3) to graphs in $\cH_{i-1}$.
Let $\cH=\bigcup_{i\ge 0}\cH_i$. 
Since Claims \ref{claim:(1)-(4)} - \ref{claim:(5) or (6)} show that operations (1) - (3) only decrease density, it follows from Theorem \ref{main good} that every $n$-vertex planar graph $G$ with $n \ge 7$ in $\cG \setminus \cH$ satisfies $e(G) \le \frac{18n}{7} - \frac{48}{7}$.
We now inductively bound the number of edges of graphs in $\cH$.

\begin{claim} \label{using small graphs}
Let $k$ be non-negative integer, and let $G$ be a $C_7$-free $n$-vertex planar graph in $\cH$ with $n\ge 3+2k$. Then either 
\begin{enumerate}[$(a)$]
\item $e(G) \le \frac{18n}{7} - \frac{30}{7} - \frac{k}{7}$, or

\item $G$ is connected and has at most two blocks (maximal connected subgraph without a cut vertex), each of which is a triangulation on five or six vertices.
\end{enumerate}
\end{claim}
\begin{proof}
Note that $n \ge 3$.
We apply induction on $n$. First, suppose $n\le 6$. 
Then $k=0$ or $k = 1$, and $G\in \cH_0$. 
If $e(G)\le 3n-7$ then $n\le 6$ implies $e(G)\le 3n-7\le \frac{18n}{7}-\frac{31}{7}$ and $(a)$ holds. So $e(G)=3n-6$ and, hence, $G$ is a triangulation.
If $n \le 4$ then $k = 0$ and $e(G) \le \frac{18n}{7} - \frac{30}{7}$.
If $n \in \{5,6\}$ then $(b)$ holds (with one block).  

So we may assume that $n \ge 7$. Then $G\notin \cH_0$, and thus, either $G$ is obtained from a graph $G_1 \in \cH$ by applying a single operation (1), or $G$ is obtained by applying a single operation (2) or (3) to graphs $G_1$ and $G_2$ in $\cH$.

First suppose that $G$ is obtained from an $n_1$-vertex graph $G_1 \in \cH$ by applying a single operation (1), and that $G$ cannot be obtained from applying operation (2) or (3) to graphs in $\cH$. 
Since $n \ge 7$ it follows that $n_1 \ge 3$, so by induction hypothesis, either $(a)$ or $(b)$ holds for $G_1$.
Suppose $(a)$ holds for $G_1$.  Thus, since $n_1\ge 3+2(k-|S|/2)$, $e(G_1) \le \frac{18n_1}{7} - \frac{30}{7} - \frac{(k-|S|/2)}{7}$. Hence, by Claim \ref{claim:(1)-(4)}, $(a)$ holds for $G$, as operation (1) gives a saving of at least $1/7$ if $|S|\in [2]$ or $2/7$ if $|S|\in \{3,4\}$. 
So assume $(b)$ holds for $G_1$.
It is a straightforward calculation to check that $e(G_1) \le \frac{18n_1}{7} - \frac{24}{7}$ (with equality when $G_1$ is a 6-vertex triangulation).
We may assume that the (18/7)-sparse set $S$ induces a connected graph, or else we may replace $S$ with one of its subsets.
We consider two cases.
If $|S| > 1$, then $S$ has at most one neighbor in $G_1$, since $G$ is $C_7$-free and $|S| \le 4$.
Since $|S| > 1$ and $G$ cannot be obtained from applying operation (2) or (3) to graphs in $\cH$, it follows that $|S| = 2$ and $S$ has no neighbor in $G_1$.
Let $G_2$ be the graph induced by $S$, which is just a single edge.
Then $e(G_2) = \frac{18n(G_2)}{7} - \frac{29}{7}$, and $G$ is obtained by applying operation (3) to $G_1$ and $G_2$.
So $n \le 13$ and $k \le 5$, and since Claim \ref{claim:(5) or (6)} gives a savings of $29/7$ it follows that $(a)$ holds for $G$.

So we may assume that $|S| = 1$, which implies that $S$ has at most two neighbors in $G_1$.
Suppose $S$ has one neighbor in $G_1$.
Then $n \le 12$ so $k \le 4$, and Claim \ref{claim:(5) or (6)} with $G_2$ as a single edge gives a savings of $11/7$, so $e(G) \le \frac{18 n}{7} - \frac{35}{7}$ and $(a)$ holds for $G$.
So we may assume that $S$ has two neighbors in $G_1$
Then since $G$ is $C_7$-free, both neighbors of $S$ are in a block of $G_1$ that is a $5$-vertex triangulation.
Since $n \ge 7$, this implies that $G_1$ has two blocks, one of which is a $5$-vertex triangulation.
Then either $n = 10$ and $e(G) = 20$ (if both blocks of $G_1$ have 5 vertices) or $n = 11$ and $e(G) = 23$ (if one block of $G_1$ has 6 vertices).
In the former case, $k \le 3$ and $e(G) \le \frac{18 n}{7} - \frac{40}{7}$, and in the latter case $k \le 4$ and $e(G) \le \frac{18 n}{7} - \frac{37}{7}$, so in either case $(a)$ holds for $G$.

Next suppose that $G$ is obtained by applying a single operation (2) or (3) to graphs $G_1$ and $G_2$ in $\cH$. For each $i\in [2]$, let $n_i=|G_i|$ and let $k_i$ be a non-negative integer such that $n_i\ge 3+2k_i$. Note that we may choose $k_1,k_2$ such that  $k_1+k_2\in 
\{k-1,k\}$. Also note that $n\ge n_1 + n_2 - 1 \ge 3+2(k_1+k_2+1)$.

First suppose that outcome $(a)$ holds for both $G_1$ and $G_2$; so for $i=1,2$, 
$e(G_i) \le \frac{18n_i}{7} - \frac{30}{7} - \frac{k_i}{7}$. 
Then Claim \ref{claim:(5) or (6)} with $c_i = \frac{30 + k_i}{7}$ for $i = 1,2$ shows that $(a)$ holds for $G$ as $k \in \{k_1 + k_2, k_1 + k_2+1\}$. 

Now suppose $(a)$ holds for $G_1$ and $(b)$ holds for $G_2$. Then Claim \ref{claim:(5) or (6)} shows that $(a)$ holds for $G$ with $k = k_1 + 2$ if $G_2$ has one block and $k = k_1 + 4$ if $G_2$ has two blocks.

So we may assume that $(b)$ holds for both $G_1$ and $G_2$. 
Suppose $(b)$ does not hold for $G$.
Then $G$ consists of 3 or 4 blocks, each of which is a triangulation with five or six vertices.
If $G$ has 4 blocks, then $n \le 24$ and so $k \le 10$.
A triangulation $B$ with 5 or 6 vertices satisfies $e(B) \le \frac{18n(B)}{7} - \frac{24}{7}$, and then applying Claim \ref{claim:(5) or (6)} three times gives a savings of at least $18/7$.
So $e(G) \le \frac{18n}{7} - \frac{42}{7}$, and since $k \le 10$ it follows that $(a)$ holds for $G$.
If $G$ is connected and has 3 blocks, then $n \le 16$ and so $k \le 6$.
Applying Claim \ref{claim:(5) or (6)} twice gives a savings of at least $12/7$, so $e(G) \le \frac{18n}{7} - \frac{36}{7}$, and $(a)$ holds for $G$ since $k \le 6$.
So $G$ is disconnected and has 3 blocks, so $n \le 18$ and $k \le 7$.
Since Claim \ref{claim:(5) or (6)} gives a savings of at least $24/7$ when operation (3) is applied, $e(G) \le \frac{18n}{7} - \frac{48}{7}$, and $(a)$ holds for $G$ since $k \le 7$.
\end{proof}

Let $G$ be an $n$-vertex $C_7$-free planar graph with $n > 38$.
If $G \in \cG \setminus \cH$ then $e(G) \le \frac{18n}{7} - \frac{48}{7}$.
If $G \in \cH$, then Claim \ref{using small graphs} applies with $k = 18$ and outcome $(a)$ holds, so $e(G) \le \frac{18n}{7} - \frac{48}{7}$, as desired.
\end{proof}

\section{Future Work} \label{sec:future work}

As discussed in Section \ref{sec:introduction}, the conjecture of Ghosh et al. that 
$\ex_{\cP}(n, C_{\ell}) \le \frac{3(\ell - 1)}{\ell}n - \frac{6(\ell + 1)}{\ell}$ for all $\ell \ge 7$ and all sufficiently large $n$ is false for all $\ell \ge 11$, but true for $\ell = 7$ by Theorem \ref{main}.
Is this conjecture true for $\ell \in \{8,9,10\}$?
We believe that the answer is yes for $\ell = 8$, and that this can be proven using existing techniques and extensive case work.
However, we believe that for $\ell = 9$, $\ex_{\cP}(n, C_{\ell}) > \frac{3(\ell - 1)}{\ell}n - \frac{6(\ell + 1)}{\ell}$ for infinitely many $n$, and we plan to prove this via a construction in a forthcoming paper.
For $\ell = 10$ the correct answer is unclear.

We mention one other direction that could be approached using techniques from this paper.
For each $\ell \ge 4$, let $\Theta_{\ell}$ denote the family of graphs (call \emph{theta graphs}) obtained by adding an edge between two non-adjacent vertices of $C_{\ell}$, and write $\ex_{\cP}(n, \Theta_{\ell})$ for the maximum number of edges in an $n$-vertex planar graph that contains no graph in $\Theta_{\ell}$ as a subgraph.
Sharp upper bounds for $\ex_{\cP}(n, \Theta_{\ell})$ are known for $\ell \in \{4,5,6\}$ by results of Lan, Shi and Song \cite{Theta-free} and Ghosh, Gy\H ori, Paulos, Xiao, and Zamora \cite{GGMPX6theta}, but no sharp upper bound is known for $\ell \ge 7$.
Due to the importance of triangular-blocks with a chord in our main proof, we believe that our techniques could make progress towards a sharp upper bound on $\ex_{\cP}(n, \Theta_{7})$.

\bibliographystyle{plain}
\bibliography{references.bib}

\end{document}